\documentclass[a4paper,reqno, 11pt]{amsart}
\usepackage{graphicx,color,marginnote}
\usepackage{footmisc}
\usepackage{amsmath}

\usepackage{amsfonts,stmaryrd}
\usepackage{amssymb}
\usepackage{bbm}
\usepackage{epstopdf}
\usepackage{color}
\numberwithin{equation}{section}

\usepackage[marginparwidth=3cm, marginparsep=.5cm]{geometry}

\usepackage{amsthm}

\newtheorem{Theorem}{Theorem}

\newtheorem{Remark}[Theorem]{Remark}
\newtheorem{Corollary}[Theorem]{Corollary}

\newtheorem{Assumption}[Theorem]{Assumption}

\newtheorem{Lemma}[Theorem]{Lemma}
\newtheorem{Definition}[Theorem]{Definition}
\newtheorem{Proposition}[Theorem]{Proposition}

\numberwithin{Theorem}{section}

\title[]{Hydrodynamic limit and viscosity solutions for a 2D  growth
  process in the anisotropic KPZ class} \author{Martin Legras}
\author{Fabio Lucio Toninelli} \address{CNRS and Institut Camille
  Jordan, Universit\'e Lyon 1, 43
  bd du 11 novembre 1918, 69622 Villeurbanne, France\\
  E-mail: legras@math.univ-lyon1.fr, toninelli@math.univ-lyon1.fr}

\begin{document}

\begin{abstract}
  We study a $(2+1)$-dimensional stochastic interface growth model,
  that is believed to belong to the so-called Anisotropic KPZ (AKPZ)
  universality class \cite{BF,BF2}. It can be seen either as a
  two-dimensional interacting particle process with drift, that
  generalizes the one-dimensional Hammersley process \cite{AD,SeppaH},
  or as an irreversible dynamics of lozenge tilings of the plane
  \cite{BF,Toninelli2+1}.  Our main result is a hydrodynamic limit: the interface
  height profile converges, after a hyperbolic scaling of space and
  time, to the solution of a non-linear first order PDE of
  Hamilton-Jacobi type with non-convex Hamiltonian (non-convexity of the Hamiltonian is
  a distinguishing feature of the AKPZ class). We prove the
  result in two situations: (i) for smooth initial profiles and times
  smaller than the time $T_{shock}$ when   singularities (shocks) appear
  or (ii) for all times, including $t>T_{shock}$, if the initial
  profile is convex. In the latter case, the height profile converges
  to the viscosity solution of the PDE. As an important ingredient, we
  introduce a Harris-type graphical construction for the process.
\end{abstract}

\maketitle
\section{Introduction}

The study of random interface growth models witnessed a spectacular
progress recently, especially in relation with the so-called KPZ
equation, cf.  e.g.  \cite{FSpo,Cor} for recent reviews.  The $d$-dimensional
discrete interface is modeled by the graph of a function $h$ from
$\mathbb Z^d$ (or some other $d-$dimensional lattice) to $\mathbb Z$.
The dynamics is an irreversible Markov chain and the asymmetry of the transition
rates by which the interface height locally increases or decreases produces a
non-trivial and slope-dependent average drift. Among the most interesting
and challenging questions are the problem of  obtaining hydrodynamic limits \cite{KL,Spohn} -- the
convergence of the height function, when space and time are rescaled
hyperbolically, to a first-order PDE -- and of understanding the
large-scale behavior of space-time correlations of the height
fluctuation process, be it in the stationary state or around the
typical macroscopic profile described by the hydrodynamic limit
equation.  Most of the known rigorous results, as far as both 
hydrodynamic limits and fluctuations are concerned, have been proven
for one-dimensional models, often in cases where the invariant
measures of interface gradients are of i.i.d. type.  Much less is known
for $d$-dimensional models, $d\ge2$ (see Section \ref{sec:commenti} for a brief overview of the literature).

When $d=2$, growth models are conjectured to fall into two distinct
universality classes, called the ``Isotropic'' and ``Anisotropic'' KPZ
classes \cite{Wolf}. (For simplicity we will refer to these two classes as KPZ and
AKPZ, respectively). For models in the KPZ class, height fluctuations
are believed (and numerically observed \cite{TFW}) to grow as a non-trivial power
$t^z$ of time $t$, as $t\to\infty$. On the other hand, for models in
the AKPZ class, height fluctuations are believed to converge asymptotically in the
scaling limit to the solution of  a stochastic heat equation with additive noise; in particular,
the variance of height fluctuations should grow like $\log t$.
Conjecturally \cite{Wolf}, the universality class a particular growth model falls
into is determined by the convexity properties of the function
$v(\rho)$ that gives the average interface drift in the
translation-invariant stationary state with slope $\rho$. Namely, it
is predicted that if the signature of the Hessian of $v(\cdot)$ is
either $(+,+)$ or $(-,-)$, then the model falls into the isotropic KPZ
class. The AKPZ class is instead relevant when the signature is
$(+,-)$ (the borderline cases, where one eigenvalue is zero and the
second is not, is also conjectured to belong to the AKPZ class).

In this work, we analyze a two-dimensional growth model in the AKPZ
class, that was originally introduced in \cite{BF,BF2} and later
studied in \cite{Toninelli2+1}.  There are various equivalent ways to
view this process. One of them is to see it as a dynamics on lozenge
tilings of the plane: the interface is then the graph of the height
function canonically associated to the tiling (see
Fig. \ref{fig:4}). Another useful viewpoint is to interpret it as a
driven, interacting system of interlaced particles that evolve on a
two-dimensional lattice via long-range jumps. Suitable sections of the
particle system behave like mutually interacting one-dimensional,
discrete, Hammersley-Aldous-Diaconis processes. In this sense, our
growth process is a two-dimensional generalization of the
Hammersley-Aldous-Diaconis process \cite{AD,SeppaH,FerrariMartin}. Yet
another viewpoint is taken in \cite{BF,BF2}, where the process is seen
as a two-dimensional totally asymmetric driven system of interacting
particles that perform only jumps of size $1$ but can ``push'' other
particles arbitrarily far away.

The main result of the present work is a hydrodynamic limit: when
space and time are rescaled hyperbolically ($x=\xi L,t=\tau L$ and
$L\to\infty$), the height function  rescaled by $1/L$
converges in probability to the solution of a first-order non-linear
PDE of Hamilton-Jacobi type. Such equations are known to have
singularities at a finite time $T_{shock}$: their gradient develops
discontinuities or shocks. Our hydrodynamic limit result is proven in
two situations: either (i) for smooth initial profile and for times up
to $T_{shock}$ (Theorem \ref{th:hl}) or (ii) for \emph{all times},
provided the initial profile is convex (Theorem \ref{th:shock}). In
the latter case, the relevant weak solution of the PDE to which the
height profile converges is the so-called viscosity solution \cite{user}.  It is
tempting to conjecture that convergence to the viscosity solution of
the PDE holds for all times also  for non-convex initial data.

The PDE describing the hydrodynamic limit is of the Hamilton-Jacobi form
\begin{eqnarray}
  \label{PDEintro}
  \partial_\tau\phi(\xi,\tau)+v(\nabla \phi(\xi,\tau))=0,
\end{eqnarray}
where the ``Hamiltonian'' (or rather ``drift'') function $v(\cdot)$ is
an explicit function (cf. \eqref{eq:v}) whose Hessian has a $(+,-)$ signature
for every slope. As we mentioned above, this is a distinguishing
feature of the AKPZ universality class but it is also a remarkable source
of difficulties, since the solution of \eqref{PDEintro} cannot be
expressed by a Hopf-Lax type formula. At the probabilistic level, lack of convexity
 prevents to prove  the hydrodynamic limit  via simple
super-additivity arguments.  More comments on these points can be
found in Section \ref{sec:commenti}.

Let us mention two previous results on our growth model, that are
relevant for the discussion here.  First, in \cite{BF} the
hydrodynamic limit was proven for a special initial condition, where
particles are perfectly packed in a certain wedge-shaped region of the
lattice. The special features of such initial conditions allow for the
application of ``integrable probability methods'' and in particular
for the exact computation, in terms of sums of determinants, of the
average particle currents integrated in time. While the same methods
may allow to analyze other ``integrable'' initial conditions (and also
more general, but still integrable, two-dimensional random growth
models, see e.g. \cite[Sec. 3.3]{BBO}), they are not robust enough to
yield ``generic'' results, where the initial microscopic configuration
is just assumed to approximate a macroscopic profile satisfying some
smoothness assumptions.  The methods we use here are of entirely
different nature with respect to those of \cite{BF} and we do not rely
on the integrable structure of the dynamics.  Let us also observe that
the initial condition chosen in \cite{BF} is such that shocks do not
appear: characteristic lines of the PDE never cross.

Secondly, a special feature of this growth process is that its stationary measures, labelled by the average slope, are known \cite{Toninelli2+1}: they coincide
with the translation invariant, ergodic Gibbs measures on lozenge tilings,
obtained as limits of uniform distributions on the torus with fixed proportions of lozenges \cite{KOS}. It is the knowledge of the stationary states that allows to obtain an explicit expression for the drift function $v(\cdot)$.

\subsection{Comments on the result and on related literature}
\label{sec:commenti}

Hydrodynamic limit results for interface growth models in dimension
$d\ge2$ are rare; here we briefly review the more relevant for us,
trying to emphasize the differences with our case. (For $d=1$ the
literature is much more vast and we refer for instance to the
introduction of \cite{Reza1} for a discussion of known results and a
list of references). In \cite{Reza1}, a class of $d$-dimensional
growth models, called v-exclusion processes there, was studied: they
enjoy a property called ``strong monotonicity'', that is stronger than
the stochastic domination or ``attractity'' property discussed e.g. in
Section \ref{sec:stochdom} of the present work. The proof of the
hydrodynamic limit in \cite{Reza1} heavily relies on super-additivity
arguments and the drift function $v(\cdot)$ in the limit PDE is
automatically convex, so that in particular these models necessarily
fall into the isotropic KPZ class. Similar ideas were used earlier
\cite{Seppadeposition} to obtain the hydrodynamic limit for a
$d$-dimensional ballistic deposition model: again, the height function
converges to the Hopf-Lax solution of a Hamilton-Jacobi equation with
convex drift function. In our context, it is relevant to mention also
the work \cite{Seppa2}, that obtains the hydrodynamic limit for a
$d$-dimensional generalization of the Hammersley-Aldous-Diaconis
process, though such generalization is quite different, even for
$d=2$, from the $2$-dimensional one we study here. Notably, once again
the drift function in \cite{Seppa2} turns out to be (strictly) convex
and, in contrast with \cite{Reza1,Seppadeposition}, it is fully
explicit. Also in \cite{Seppa2}, the proof of convergence to the Hopf-Lax
solution of the PDE uses super-additivity.

Another very interesting work is \cite{Reza2}: the growth models
studied there do not satisfy strongly monotonicity. For dimension
$d\ge2$, however, the hydrodynamic limit obtained there is weak in the
sense that it is not known that the drift function $v(\cdot)$ is
non-random (i.e., the height function could converge in the scaling
limit to the solution of a \emph{random} first order Hamilton-Jacobi
equation).

The Markov chain we study in the present work is rather different from
those just mentioned. For one thing, it does not satisfy strong
monotonicity, super-additive arguments do not work and, as we already
noted, the drift function $v(\cdot)$ in \eqref{PDEintro} turns out to
have no definite convexity or concavity. What comes to rescue is the
knowledge of the stationary measures, that allows to compute
$v(\cdot)$. Lack of convexity of $v(\cdot)$ induces serious analytic
problems in the analysis of the PDE. Notably, we are not aware of any
variational formula of the Hopf-Lax type expressing its solution, for
general initial condition. This is one of the main reasons why we
cannot prove convergence to the hydrodynamic limit for general initial
profiles and for times $t>T_{shock}$.  A Hopf-type variational formula
\emph{is} however available when the initial datum is convex or
concave \cite{BE}, and this is strongly used in the proof of Theorem
\ref{th:shock} below.

Let us conclude this introduction with a few more comments on the
peculiar technical difficulties one encounters in the proof of the
hydrodynamic limit for our growth process. An important ingredient in
the methods of \cite{Reza2}, that is a key step to obtain a tightness
property in Skorohod space for the height function process, is that,
uniformly with respect to the initial condition, the
height at a given point can grow at most linearly in time. This is
definitely false for our model. In fact, the average speed of growth
is larger and larger as the particles are more and more mutually
spaced.  This can be seen, at the macroscopic level, by the fact that
the function $v(\cdot)$ in \eqref{eq:v} diverges as $\rho_1+\rho_2$
tends to $1$ (as we will see, this corresponds exactly to the
situation where particle spacings diverge).  One might hope that, if
particle spacings are tight in the initial condition, then the same
holds at all later times.  For instance, a stochastic domination
property of the following type would be very helpful: given two
configurations such that the former has larger particle spacings than
the latter, the evolutions with the two initial conditions can be
coupled in a way that the same property holds at all later times.  In
fact, a property of this type is valid for the \emph{one-dimensional}
Hammersley-Aldous-Diaconis process \cite[Sec. 6]{SeppaH}, but
unfortunately it seems to fail for our two-dimensional model. The estimates on
``speed of propagation of information'' and the ``localization
procedure'' introduced in Section \ref{sec:importante}, as well as the
iterative procedure of Sections \ref{sec:hl} and \ref{sec:shocks}, are
devised precisely to overcome these difficulties.  Finally, let us
emphasize that an important tool for our proof, as was the case for
the Hammersley-Aldous-Diaconis processes in \cite{SeppaH}, is a
formulation of the growth process in terms of a suitable graphical
construction.

\subsection*{Organization of the article}
In Section \ref{sec:model} we define the state space of the
interacting particle model and the associated height function;
moreover, we introduce the growth process via a graphical construction
and we state some of its basic properties (Proposition
\ref{prop:bendefinito}, that is proven in Section
\ref{sec:bendefinito}). In Section \ref{sec:mainresult} we state our
main result: the hydrodynamic limit for the height function (Theorems
\ref{th:hl} and \ref{th:shock}).  In Section \ref{sec:importante} we
prove two crucial properties of the dynamics: monotonicity and a bound
on the speed of propagation of information. Finally, Theorems
\ref{th:hl} and \ref{th:shock} are proved in Sections \ref{sec:hydro1}
and \ref{sec:shocks} respectively.

\section{Model}
\label{sec:model}
\subsection{Configuration space and informal description of the dynamics}
\label{sec:confspace}
 The lattice where particles live consists of an infinite
collection of horizontal lines, labeled by an index $\ell\in \mathbb
Z$. Each line contains an infinite collection of particles, each
having a label $(p,\ell)$, $p\in \mathbb Z$. See Figure \ref{fig:1}. Horizontal particle
positions $z_{(p,\ell)}$ are discrete and
\[
z_{(p,\ell)}\in\mathbb Z+(\ell\!\!\!\!\mod 2)/2:
\] on lines with index
$\ell\in 2\mathbb Z$ one has $z_{(p,\ell)}\in \mathbb Z$, while on
lines with index $\ell\in 2\mathbb Z+1$ one has $z_{(p,\ell)}\in
\mathbb Z+1/2$ (the reason for this choice is that it will be
convenient that no two particles in neighboring lines have the
same horizontal position).  % For convenience, we let
% \[
% \mathbb Z_\ell:=\mathbb Z+(\ell\!\!\!\!\mod 2)/2.
% \]
% PAS TERRIBLE PEUT ETRE COMME NOTATION...

Moreover, particle positions satisfy a number of constraints:
\begin{Definition}
  \label{def:omega}We let $\Omega$ be the set of  particle configurations $\eta$ satisfying the following properties:
\begin{enumerate}
\item no two particles in the same line $\ell$ have the same position
  $z_{(p,\ell)}$. We can then label particles in each line in such a
  way that $z_{(p,\ell)}<z_{(p+1,\ell)}$. Labels should be seen as
  attached to particles, and they will not change along the dynamics.
\item particles are interlaced in the following sense: for every
  $\ell$ and $p$, there exists a unique $p'\in \mathbb Z$ such that
  $z_{(p,\ell)}<z_{(p',\ell+1)}<z_{(p+1,\ell)}$ (and, as a
  consequence, also a unique $p''\in \mathbb Z$ such that
  $z_{(p,\ell)}<z_{(p'',\ell-1)}<z_{(p+1,\ell)}$). Without loss of
  generality, we will assume that $p'=p$ (and therefore
  $p''=p+1$). This can always be achieved by deciding which particle
  is labeled $0$ on each line.  Also, by convention, we establish that the particle labeled $(0,0)$ is the left-most one on line $\ell=0$, with non-negative  horizontal coordinate.
\item for every $\ell$ one has
\begin{eqnarray}
  \label{eq:1}
 \lim_{p\to -\infty} \frac{z_{(p,\ell)}}{p^2}=0.
\end{eqnarray}
\end{enumerate}
\end{Definition}
Note that if condition \eqref{eq:1} holds for $\ell=0$ it holds for every $\ell$.
% In order for the dynamic to be well defined, we need to make sure that no particule will "jump" an infinite distance in a finite time, so we restrict the set to be :

% \[F:=\{z\in E'|\forall l\in Z, \lim_{p\to \infty} \frac{z_{(p,\ell)}}{p^2}=0\} \]    
%C'EST UNE NOTATION PLUS STANDARD POUR LES CHAINES DE MARKOV, EN PLUS E PEUT ETRE CONFUS AVEC UNE ESPERANCE.

% $q:=(p,l)$ is said to be a \textit{particule label}, $z_q$ denotes the position of $p$-th particule of the $l$-th line.

Here we give an informal description of the dynamics. A more rigorous version is given in Section \ref{sec:graph}. The property (3) in Definition \ref{def:omega} will ensure that the dynamics is well-defined.

Given the particle labelled $(p,\ell)$, we denote $I_{(p,\ell)}=\{(p-1,\ell+1),(p,\ell-1)\}$, see Fig. \ref{fig:1}.  Note that
these are simply the labels of the two particles directly to the left of $(p,\ell)$
on lines $\ell+1$ and $\ell-1$.
\begin{figure}
\includegraphics[width=9cm]{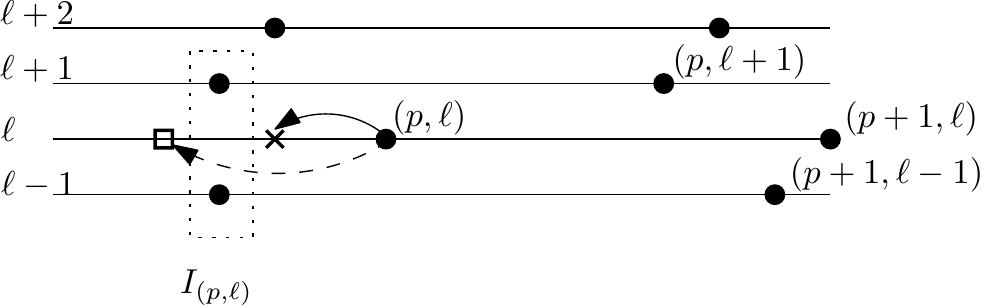}
  \caption{The set $I_{(p,\ell)}$ comprises the particle labels of the two particles in the dotted region. The particle $(p,\ell)$ can jump  to the position marked by the cross (and does so with rate $1$) but cannot jump to the position marked by a square, because
it would have to overcome the particles in $I_{(p,\ell)}$.  }
\label{fig:1}
\end{figure}
  To every pair $(\ell,z)$ with
$\ell\in\mathbb Z$ and $z\in \mathbb Z+(\ell\!\!\mod 2)/2$ we
associate an i.i.d. Poisson clock of rate $1$.   When the clock labeled
$(\ell,z)$ rings, then:
\begin{itemize}
\item if position $(\ell,z)$ is occupied, i.e. if there is a particle
  on line $\ell$ with horizontal position $z$, then nothing happens;
\item if position $(\ell,z)$ is free, let $(p,\ell)$ denote the label
  of the left-most particle on line $\ell$, with $z_{(p,\ell)}>z$. If
  both particles in $I_{(p,\ell)}$ have horizontal position smaller
  than $z$, then particle $(p,\ell)$ is moved to position $(\ell,z)$;
  otherwise, nothing happens.
\end{itemize}
The second point can be described more compactly as follows: is position $(\ell,z)$ is free, then particle $(p,\ell)$ is moved to position  $(\ell,z)$ if and only if the new configuration is still in $\Omega$, i.e. if the interlacement constraints are still satisfied.

\begin{Remark}
  Recall the definition \cite{FerrariMartin} of the one-dimensional (discrete)
  Hammersley-Aldous-Diaconis (HAD) process: on each site $x\in\mathbb Z$ there is at most
  one particle; each particle jumps with rate $1$ to any position that
  is at the same time to its left and to the right the next particle.
  Going back to our two-dimensional interacting particle system, we
  see that particles on each line $\ell$ follow a HAD process,
  except that particle jumps can be prevented by the interlacing
  constraints with particles in lines $\ell\pm1$. This induces
  non-trivial correlations between the processes on different lines.
  As discussed in Section \ref{sec:stationary}, the invariant measures
  of the two-dimensional dynamics restricted to any line $\ell$ are
  very different from the (i.i.d. Bernoulli) invariant measures of the
  HAD process.
\end{Remark}

\subsection{Height function}
\label{sec:height}
To each configuration $\eta\in\Omega$ we associate an integer-valued
height function $h_\eta$. We first give the definition, then we
motivate it via the bijection between particle configurations
and lozenge tilings of the plane.

First of all let us remark that  the graph $G$ whose vertices are  all the possible particle positions $(\ell,z),\ell\in\mathbb Z, z\in \mathbb Z+(\ell\!\!\mod 2)/2$ and where the neighbors of $(\ell,z)$ are 
the four vertices $(\ell\pm1,z\pm1/2)$ can be identified with $\mathbb Z^2$, rotated by $\pi/4$ and suitably rescaled, see Figure \ref{fig:2}.
\begin{figure}
\includegraphics[width=11cm]{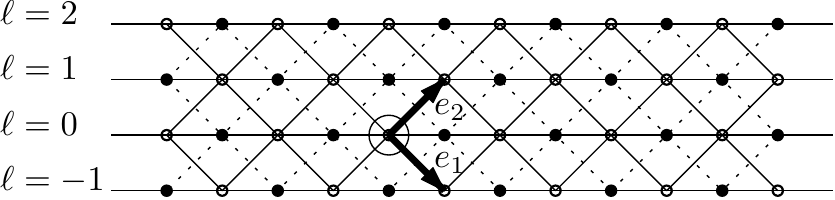}
\caption{The lattice $G$ (dotted, with black vertices that are
  possible particle positions $(\ell,z)$, with
  $z\in\mathbb Z+(\ell\!\!\mod 2)/2$), the lattice $G^*$ (bold, with white
  vertices) and the two coordinate vectors $e_1,e_2$. The origin of
  $G^*$ (encircled) is the point of horizontal coordinate $z=-1/2$ on
  line $\ell=0$.}
\label{fig:2}
\end{figure}The height function $h_\eta$ is defined on the dual graph
$G^*$, that is just obtained by shifting $G$ horizontally by $1/2$,
see  Figure \ref{fig:2}. In other words, on each line $\ell$, the
  height function is defined at horizontal coordinates
  $n\in \mathbb Z+((\ell+1)\!\!\mod 2)/2$.

There is a natural and convenient choice of coordinates on $G^*$:
\begin{Definition}[Coordinates on $G^*$]
\label{def:coord}
 The point of $G^*$ of horizontal coordinate $-1/2$ of the line
labeled $\ell=0$ is assigned the coordinates $(x_1,x_2)=(0,0)$. The unit vector
$e_1$ (resp. $e_2$) is the vector from $(0,0)$ to the point of horizontal coordinate $0$ on the line labeled $\ell=-1$
(resp. $\ell=+1$), see Figure \ref{fig:2}. With this convention, the vertex of $G^*$ labeled $(x_1,x_2)$ is on line 
\begin{equation}
  \label{eq:elle}
\bar\ell(x)=x_2-x_1  
\end{equation}
 and has
horizontal coordinate 
\begin{equation}
  \label{eq:zeta}
\bar z(x)=(x_1+x_2-1)/2.   
\end{equation}
\end{Definition}
We can now define the height function $h_\eta$:
% See Figure ... %%  Note that the collection  $G^*$ of
%% points where the height is defined is naturally seen as $\mathbb
%% Z^2$, turned by an angle $\pi/4$.
\begin{Definition}[Height function]
\label{def:height}
Given a configuration $\eta\in\Omega$, its height function $h_\eta$ is
an integer-valued function defined on $G^*$.  We fix
$h_\eta(0,0)$ to some constant (e.g. to zero) and then it is enough to define the gradients
$h_\eta(x_1+1,x_2)-h_\eta(x_1,x_2)$ and
$h_\eta(x_2,x_2+1)-h_\eta(x_1,x_2)$ to fix $h_\eta$ unambiguously.

Given $(x_1,x_2)\in G^*$, let $p$ (resp. $p+1$) be the index of the rightmost
(resp. leftmost) particle on line $\ell=x_2-x_1$ that is to the left\footnote{Recall that particles are on $G$ and not on $G^*$: therefore, a particle to the left (resp. right) of $(x_1,x_2)\in G^*$ is strictly to the left (resp. right) of it.  }
(resp. to the right) of $(x_1,x_2)$. Recall that particle $p$ of line
$\ell+1$ satisfies $z_{p,\ell}<z_{p,\ell+1}<z_{p+1,\ell}$. We
establish that
\begin{multline}
  \label{eq:height}
\Delta_2 h_\eta(x_1,x_2):= h_\eta(x_1,x_2+1)-h_\eta(x_1,x_2)\\=\left\{
    \begin{array}{ccc}
      0 & \text{if} & (x_1,x_2+1) \text{ is to the right of particle } (p,\ell+1)\\
      1 & \text{if} & (x_1,x_2+1) \text{ is to the left of particle } (p,\ell+1)
    \end{array}
\right.
\end{multline}
and similarly 
\begin{multline}
  \label{eq:height2}
\Delta_1 h_\eta(x_1,x_2):=  h_\eta(x_1+1,x_2)-h_\eta(x_1,x_2)\\=\left\{
    \begin{array}{ccc}
      0 & \text{if} & (x_1+1,x_2) \text{ is to the right of particle } (p+1,\ell-1)\\
      1 & \text{if} & (x_1+1,x_2) \text{ is to the left of particle } (p+1,\ell-1).
    \end{array}
\right.
\end{multline}

\end{Definition}
See Figure \ref{fig:3}. 
\begin{figure}
\includegraphics[width=11cm]{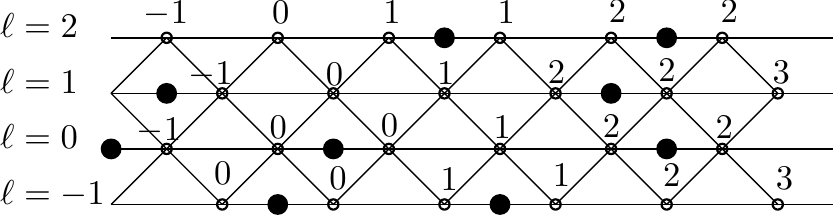}
\caption{A portion of particle configuration and the corresponding height function. Particles (black dots) are on vertices of $G$ and the height is defined on vertices of $G^*$.}
\label{fig:3}
\end{figure}
We leave it to the reader to check that
\begin{multline}
  \label{eq:grad2}
\Delta_1 h_\eta(x_1,x_2)+\Delta_2 h_\eta(x_1+1,x_2)=\Delta_2 h_\eta(x_1,x_2)+\Delta_1 h_\eta(x_1,x_2+1)\\
=h_\eta(x_1+1,x_2+1)-h_\eta(x_1,x_2)=\left\{
    \begin{array}{ll}
      0 & \text{if $\exists$    particle between } (x_1,x_2) \text{ and } (x_1+1,x_2+1)\\
      1 & \text{if $\nexists$    particle between } (x_1,x_2) \text{ and } (x_1+1,x_2+1).
    \end{array}
\right.
\end{multline}
The first equality in \eqref{eq:grad2} implies that the sum of gradients of $h_\eta$ along any closed circuit is zero, so that the definition of $h_\eta$ is well-posed.
 
Let $\mathbb T\subset \mathbb R^2$ denote the closed triangle with
vertices $(0,0),(1,0),(0,1)$ and $\stackrel\circ{\mathbb T}$ be its
interior.  The following holds:
\begin{Lemma}
\label{lemma:approssimazione}
 Let $A\subset \stackrel\circ{\mathbb T}$.  Let $\phi:\mathbb R^2\mapsto \mathbb R$ be a Lipschitz function satisfying  $\nabla \phi(x)\in A$ for almost every $x\in \mathbb R^2$. There exists a sequence $\{\eta^{(L)}\}_{L\in\mathbb N}$ in $\Omega$ such that  
 \begin{eqnarray}
   \label{eq:approssimazione}
|   h_{\eta}(x)-L \phi(x/L)|\le 1 \quad \text{ for every } \quad  x\in \mathbb Z^2.
 \end{eqnarray}
\end{Lemma} 
The proof of Lemma
  \ref{lemma:approssimazione} is given at the end of next section
  since it is more immediate after discussing the mapping between
  particle configurations and lozenge tilings.

  The restriction $\nabla \phi(x)\in \mathbb T$ can be easily
  understood as follows. At the microscopic level, the interface
  gradients
  \begin{eqnarray}
    \label{eq:gradmicro}
    h_\eta(x+(1,0))-h_\eta(x), \;h_\eta(x+(0,1))-h_\eta(x),\; h_\eta(x+(1,1))-h_\eta(x)
  \end{eqnarray}
  all belong to $\{0,1\}$. Therefore, any Lipschitz interface $\phi$ that
  can be approximated by a discrete one must verify
  $\partial_{x_i}\phi\in [0,1]$, as well as
  $\partial_{x_1}\phi+\partial_{x_2}\phi \in [0,1]$. This is exactly the
  condition $\nabla \phi\in \mathbb T$. 

\subsubsection{Height function and mapping to lozenge tilings}

In order to understand the definition of height function given above,
let us first of all recall that there is a bijection between
interlaced particle configurations satisfying properties (1)-(2)
of Definition \ref{def:omega} and lozenge tilings of the plane, as in
Figure \ref{fig:4}. 
\begin{figure}
\includegraphics[width=13cm]{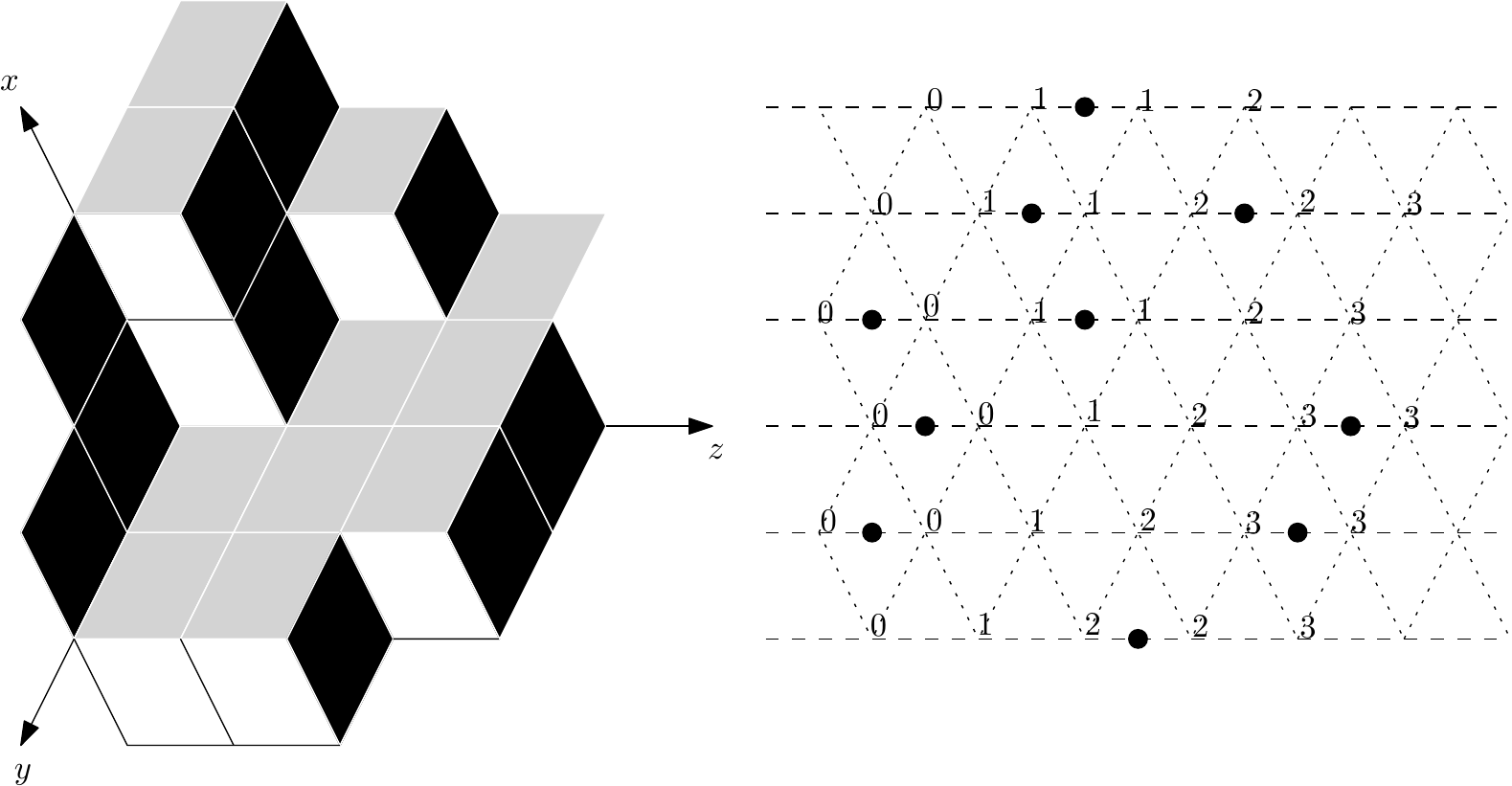}
\caption{The bijection between lozenge tiling (or stepped interface)
  and particle configuration: vertical (black) lozenges correspond to
  particles (dots). The reader may check that the height function in
  the right drawing corresponds to the height in the $z$ direction, that is  w.r.t. the $(x,y)$ plane,
  of the pile of cubes in the left drawing. }
\label{fig:4}
\end{figure}
 Particles correspond to vertical lozenges: the vertical
coordinate of the central point of a vertical lozenge defines the line the particle is on,
and its horizontal coordinate corresponds to the $z_{(p,\ell)}$
coordinate of the particle. Note that, if lengths are rescaled in such
a way that lozenge sides are $1$, then horizontal positions of
vertical lozenges are shifted by half-integers between neighboring
lines (as is the case for particles). That horizontal positions of
lozenges in neighboring lines satisfy the same interlacings as
particle positions $z_{(p,\ell)}$, as well as the fact that the
tiling-to-particle configuration mapping is a bijection, is well known
and easy to understand from the picture. 

 Given a lozenge tiling as in Figure \ref{fig:4} and
viewing it as the boundary of a stacking of unit cubes in
$\mathbb R^3$, a natural definition of height function is to assign to
each vertex of a lozenge the height (i.e. the $z$ coordinate) w.r.t. the $(x,y)$ plane of the
point (in $\mathbb R^3$) in the corresponding unit cube. As a consequence, height is
integer-valued and defined on points that are horizontally shifted
$1/2$ w.r.t. centers of lozenges, i.e. on points of $G^*$.  We leave
to the reader to check that the height function as defined in the
previous section precisely corresponds to the height of the stack of
cubes w.r.t. the $(x,y)$ plane.

\begin{proof}
  [Proof of Lemma \ref{lemma:approssimazione}]
Define 
\begin{eqnarray}
  \label{eq:happr}
  h(x):=\lfloor L \phi(x/L)\rceil, \,x\in\mathbb Z.
\end{eqnarray}
We want first to prove that such function $h(\cdot)$ is the height function of some lozenge tiling.
Given $x$, let 
\begin{multline}
n_1(x)=h(x+(1,0))-h(x),\\ n_2(x)=h(x+(1,1))-h(x+(1,0), \\n_3(x)=1-[h(x+(1,1))-h(x)].  
\end{multline}
By the assumption $\nabla \phi\in \mathbb T$, we see that
$n_i(x)\in\{0,1\}$. Also, since $n_1(x)+n_2(x)+n_3(x)=1$, we have that
necessarily two of them are $0$ and the third is $1$.  Draw an edge
between $x$ and $x+(1,0)$ (resp. between $x+(1,0)$ and $x+(1,1)$,
resp. between $x$ and $x+(1,1)$) if $n_1(x)=0$ (resp. if $n_2(x)=0$,
resp. if $n_3(x)=0$) and no edge otherwise.  Therefore, we see that
the set of edges defines a lozenge tiling of the plane (every
elementary triangle $x,x+(1,0),x+(1,1)$ has two edges on the boundary
of a lozenge and the third crosses a lozenge) and it is clear that $h$
is just the height function of the corresponding cube stacking.

To prove that $h(\cdot)$ in \eqref{eq:happr} is the height function of a particle configuration in $\Omega$ it remains to show that \eqref{eq:1} is satisfied.
Since $\nabla \phi\in A\subset \stackrel\circ{\mathbb T}$, we have that $\partial_{x_1}\phi+\partial_{x_2}\phi$ is bounded away from $1$. From \eqref{eq:grad2} we deduce that 
inter-particle distances $z_{(p,\ell)}-z_{(p-1,\ell)}$ are uniformly bounded, so that $z_{(p,\ell)}$ grows linearly for $p\to-\infty$.
\end{proof}

\subsection{Graphical definition of the dynamics}
\label{sec:graph}
Here we give a precise definition of the dynamics informally
introduced in Section \ref{sec:confspace}.  Given an initial condition
$\eta\in \Omega$ and $t>0$, we wish to define particle positions
$z_{(p,\ell)}(t)$ at time $t$. In Proposition \ref{prop:bendefinito}
we will see that the dynamics satisfies a Markov-type semi-group
property.

We need a few notations.  For every pair $(\ell,z)$ with
$z\in \mathbb Z+(\ell\!\!\mod 2)/2$, let $\mathcal W_{(\ell,z)}$ be a
Poisson point process on $\mathbb R^+$, of intensity $1$. Poisson
processes with different $\ell$ or $z$ are independent. One should
view $\mathcal W_{(\ell,z)}$ as the times the exponential clock
located on line $\ell$ at horizontal position $z$ rings. We denote
also $\mathcal W_\ell$ the collection of
$\mathcal W_{(\ell,z)},z\in \mathbb Z+(\ell\!\!\mod 2)/2$ and
$\mathcal W$ the collection of all
$\mathcal W_{(\ell,z)},\ell\in \mathbb Z,z\in \mathbb Z+(\ell\!\!\mod
2)/2$. The law of $\mathcal W$ is denoted $\mathbb P$.

\begin{Definition}
\label{def:Xi}
Given $t>0$, a realization $W$ of $\mathcal W$, a configuration
$\eta\in \Omega$ with particle positions denoted $\{z_{(p',\ell')}\}_{p',\ell'}$ and a
particle label $(p,\ell)$, we let $\Xi_{(p,\ell),\eta,W,t}$ be the set
of all finite (possibly empty) collections $\xi$ of points
$w=((x,s),(p',\ell'))$, satisfying the conditions listed below. Given
$w=((x,s),(p',\ell'))\in \xi$, we call $x$ its space label, $s$ its
time label, and $(p',\ell')$ its particle label.  We impose the
following constraints on $\xi$:
\begin{enumerate}
\item [(I)] If $((x,s),(p',\ell'))\in \xi$ then:
  \begin{itemize}
  \item $0\le s\le t$;
  \item the
  labels $p'$ and $\ell'$ are in $\mathbb Z$ while $x\in \mathbb
  Z+(\ell'\!\!\!\mod 2)/2$;
\item  the realization $W_{(\ell',x)}$  of the Poisson  process $\mathcal W_{(\ell',x)}$  contains the point $s$;
 \end{itemize}
   
%\item [(II)] If $((x,s),(p',\ell'))$ belongs to $\xi$, then 

\item [(II)] If $\xi\ne\emptyset$, there is a unique  point in $\xi$ with largest space label $x$ and this point has particle label $(p,\ell)$. We note $x_0(\xi)$ the space label of this point. If $\xi=\emptyset$, we set $x_0(\xi):=z_{(p,\ell)}$;
\item [(III)] If $((x,s),(p',\ell'))\in \xi$ and if there exists a
  particle with label $(p'',\ell'')\in I_{(p',\ell')}$ such
  that\footnote{since $x\in \mathbb Z+(\ell'\mod 2)/2$ while
    $z_{(p'',\ell'')}\in \mathbb Z+(\ell''\mod 2)/2$ with
    $|\ell'-\ell''|=1$, we have $|x-z_{(p'',\ell'')}|\ge 1/2$ so the inequality
    $ x \le z_{(p'',\ell'')}$ is strict, if it holds}
  $ x < z_{(p'',\ell'')}$, then there exists a point
  $((x',s'),(p'',\ell''))$ in $\xi$ such that $x'< x,s'\le s$;
\item [(IV)] The point  $((x',s'),(p'',\ell''))$ in the previous item is unique;

\item  [(V)] Conversely, if $((x',s'),(p'',\ell''))\in \xi$ then either
  $x'=x_0(\xi)$ or there exists \[((x,s),(p',\ell'))\in \xi\] with
  $(p'',\ell'')\in I_{(p',\ell')}$, $ x < z_{(p'',\ell'')}$ and such that $x>x',s\ge s'$.

\end{enumerate}
\end{Definition}

\begin{Remark}
\label{rem:II}
Note that, thanks to conditions (II) and (V), in every
$\emptyset\neq\xi\in\Xi_{(p,\ell),\eta,W,t}$ there is a unique point with particle
label $(p,\ell)$. Moreover, such point has the largest time and space coordinates.
\end{Remark}
% For a given $P$, a "time" $t\in R_+$, a configuration $\eta \in F$ and a particule label $q=(p,l)$ we construct $\Xi_{(p,\ell),\eta,W,t}$. Each $w\in\Omega_{p,l,\eta,P}$ is a finite collection of points $(x,s,q')$ with $(x,s)$ a space time-position and $q'=(p',l')$ a particule label, such that :

We can now define our dynamics: for any initial condition $\eta\in\Omega$, any time $t>0$, any particle label $(p,\ell)$ and any realization $W$ of the Poisson point processes $\mathcal W$ we let
\begin{eqnarray}
  \label{eq:3}
z_{(p,\ell)}(t):=\inf_{\xi\in
    \Xi_{(p,\ell),\eta,W,t}} x_0(\xi).
\end{eqnarray}
Note that $z_{(p,\ell)}(0)=z_{(p,\ell)}$ for almost every realization
of the Poisson processes, since almost surely there is no Poisson
point with time coordinate equal to $0$. Also,
$z_{(p,\ell)}(t)\le z_{(p,\ell)}$ (particles move to the left) since
$\Xi_{(p,\ell),\eta,W,t}$ contains the empty set.

In Section \ref{sec:bendefinito} we will prove the following:
\begin{Proposition}
\label{prop:bendefinito} Let $\eta(t)$ denote the configuration where every 
particle $(p,\ell)$ has horizontal position $z_{(p,\ell)}(t)$ defined
by \eqref{eq:3}.  The following holds for almost every realization $W$ of
the Poisson process:
  \begin{enumerate}
 \item for every $(p,\ell)$,  the infimum in \eqref{eq:3} is actually a minimum for every $t>0$.
  \item  for every $t>0$, $\eta(t)\in\Omega$. 
 
  \item the following semi-group property holds: for every $0\le s<t$, for every $(p,\ell)$, 
    \begin{eqnarray}
      \label{eq:semigruppo}
z_{(p,\ell)}(t)=\inf_{\xi\in \Xi_{(p,\ell),\eta(s),\theta_s(W),t-s}} x_0(\xi),
    \end{eqnarray}
with $\theta_s(W)$ the time-translation by $-s$ of $W$. 
\item  for every $(p,\ell)$,
  \begin{eqnarray}
\mathbb P(\exists \xi\in \Xi_{(p,\ell),\eta,W,t}, |\xi|\ge 2 \text{ and }
\xi \text{ realizes the infimum in \eqref{eq:3}})\stackrel{t\to0}=O(t^2).
  \end{eqnarray}
  \end{enumerate}
\end{Proposition}

\begin{Remark}
  We can now convince the reader that formula \eqref{eq:3} does indeed
  correspond to the informal definition of the dynamics given in
  Section \ref{sec:confspace}. By the ``semi-group'' property
  \eqref{eq:semigruppo}, it is sufficient to consider an infinitesimal
  time interval $[0,\delta]$.  Consider the particle labelled $(p,\ell)$ and let
  $z^-:=\max\{z_{(q,\ell')},(q,\ell')\in I_{(p,\ell)}\}$. By point (4)
  in Proposition \ref{prop:bendefinito}, the probability that $\xi$
  realizing the infimum in \eqref{eq:3} contains more than one point
  is $O(\delta^2)$ and can therefore be disregarded for the computation of the transition rates. If a set
  $\xi\in\Xi_{(p,\ell),\eta,W,\delta}$ contains exactly one point $P$ and
  $x_0(\xi)<z_{(p,\ell)}$, then $P$ must have time coordinate in
  $[0,\delta]$ and its space coordinate $n=x_0(\xi)$ can take any of the
  % $K:=z_{(p,\ell)}-z^--1/2\in \mathbb N\cup \{0\}$ 
   values in $\mathbb Z+(\ell\!\!\mod 2)/2$ that are
  strictly between $z^-$ and $z_{(p,\ell)}$. The probability that
  there is a Poisson point before time $\delta$ at a given position $n\in (z^-,z_{(p,\ell)})$
  is $\delta+O(\delta^2)$, i.e. particle $(p,\ell)$ jumps to that
  position with rate $1$ and cannot jump to the left of $z^-$. % (Note that  with probability $1-\delta K +O(\delta^2)$, the
  % minimizer in \eqref{eq:3} is the empty set and
  % $z_{(p,\ell)}(t)=z_{(p,\ell)}(0)$.)
This is exactly the informal description of the dynamics we gave above.  
\end{Remark}

\section{Main result: hydrodynamic limit}
\label{sec:mainresult}

 Recall the way the
height function was defined in Section \ref{sec:height}.  Given
$x=(x_1,x_2)\in G^*$ and $t>0$, we let
\begin{eqnarray}
  \label{eq:H}
 H(x,t)=h_{\eta}(x)-J_x(t)
\end{eqnarray}
where $\eta\in\Omega$ is the initial particle configuration
while $J_x(t)$ is the number of particles, on line labelled
$\ell=x_2-x_1$, that cross point $x$ (from right to left) in the time
interval $[0,t]$.  
Note that $H(\cdot,t)$ is nothing but the height function $h_{\eta(t)}$ of the configuration $\eta(t)$ (according to the definition in Section \ref{sec:height}), \emph{up to a global additive constant} $H(0,t)=-J_0(t)\le 0$.

In order to formulate a hydrodynamic limit theorem
we need, rather than a single initial condition, a sequence
$\{\eta^{(L)}\}_{L\in\mathbb N}$ of initial conditions, that
approximate a smooth profile:
\begin{Assumption}
\label{assumption:1}
   The initial condition $\eta^{(L)}\in\Omega$  is such that 
\begin{eqnarray}
  \label{eq:condiniz}
h_{\eta^{(L)}}(x)=\lfloor L \phi_0(x/L)\rfloor \quad \text{for every }\quad x\in \mathbb Z^2,
\end{eqnarray}
where $\phi_0:\mathbb R^2\mapsto \mathbb R$ is a Lipschitz function
such that $\nabla\phi_0(x)\in A$ for almost every $x\in\mathbb R^2$,
where $A$ is a compact subset of $\stackrel\circ{\mathbb T}$.
 \end{Assumption}

 \begin{Remark}
   Recall from the proof of Lemma \ref{lemma:approssimazione} that
   \eqref{eq:condiniz} indeed defines an admissible particle
   configuration in $\Omega$.    One could allow
   for initial conditions $\eta^{(L)}$ whose height function
   approximates $\phi_0$ in a less strong sense, but some uniformity
   in space is needed and it would not be enough for our purposes
   to require just that 
   $(1/L) h_{\eta^{(L)}}(\lfloor x L\rfloor)$ tends pointwise to $\phi_0(x)$ as $L\to\infty$.
 \end{Remark}

We also  define, for $\rho=(\rho_1,\rho_2)\in\stackrel\circ{\mathbb T}$, the ``drift function''
\begin{eqnarray}
  \label{eq:v}
  v(\rho)=\frac1\pi\frac{\sin(\pi \rho_1)\sin(\pi \rho_2)}{\sin(\pi(\rho_1+\rho_2))}>0.
\end{eqnarray}
Note that $v(\cdot)$ is $C^\infty$ in $\stackrel\circ{\mathbb T}$;
also, as was observed in \cite{BF,BF2}, for every
$\rho\in\stackrel\circ{\mathbb T}$ the Hessian matrix
$\{\partial^2_{\rho_i \rho_j}v(\rho)\}_{i,j=1,2}$ has one strictly
positive and one strictly negative eigenvalue.

Our first result assumes further smoothness properties of the initial
profile $\phi_0$:
\begin{Assumption}
  The initial profile $\phi_0$ is $C^2$ and $\sup_{x\in \mathbb R^2}\|H_{\phi_0}(x)\|<\infty$, with $H_{\phi_0}(x)$ the Hessian matrix of $\phi_0$ at $x$. In addition, $\nabla\phi_0(x)\in A$ for every $x$, with $A$  a compact subset of $\stackrel\circ{\mathbb T}$.
\label{assumption:2}
\end{Assumption}

We start with a rather standard fact (see Section \ref{sec:PDE} for a few details):
\begin{Proposition}
\label{prop:PDE}
Let $\phi_0$ satisfy Assumption \ref{assumption:2}. There exists $T>0$
such that the PDE
  \begin{eqnarray}
    \label{eq:PDE}
\left\{
  \begin{array}{c}
    \partial_t \phi(x,t)+v(\nabla \phi(x,t))=0 \\
\phi(x,0)=\phi_0(x)
  \end{array}
\right.
  \end{eqnarray}
  has a twice differentiable classical solution for $(x,t)\in\mathbb R^2\times(0,T]$
  and \[\sup_{x\in\mathbb R^2, t\le
    T}\left\{\|H_{\phi(x,t)}\|+|\partial^2_t\phi(x,t)|\right\}<\infty.\] 
Moreover, 
\begin{eqnarray}
  \label{eq:more}
  \nabla\phi(x,t)\in A \text{ for every }x\in \mathbb R^2, t\le T.
\end{eqnarray}
\end{Proposition}
Actually, if we let $T_f$ denote the supremum of the values of
$T$ such that the above holds, then $T_f$ coincides with the first
time the method of characteristics ceases to work for  equation \eqref{eq:PDE}, or more precisely the
first time the application $x_0(\cdot,\cdot)$, that to $(x,t)$ associates the starting
point $x_0(x,t)$ of the characteristic line going through $(x,t)$,
ceases to have a Jacobian determinant that is everywhere non-zero. See Section \ref{sec:PDE}.

Our first hydrodynamic limit result is then:
\begin{Theorem}
\label{th:hl}
Let the initial configuration $\eta^{(L)}$ satisfy Assumption \ref{assumption:1}, with $\phi_0$ satisfying Assumption \ref{assumption:2}.
For all $t<T_f,x\in \mathbb R^2$ and for any $\delta>0$ we have
\begin{eqnarray}
  \label{eq:convergenza}
\lim_{L\to\infty}  \mathbb P\left(\left|\frac1L H(\lfloor xL\rfloor, tL)-\phi(x,t)\right|>\delta\right)=0.
\end{eqnarray}
\end{Theorem}

\subsection{Beyond the singularities: hydrodynamic limit in presence of shocks}

In general, the PDE \eqref{eq:PDE} develops singularities (more
precisely, discontinuities of the gradient) in finite time, even for
smooth initial data. A very interesting question is whether there
exists a weak solution of \eqref{eq:PDE} such that the convergence
\eqref{eq:convergenza} holds even after the time when singularities
arise: a natural candidate is the so-called viscosity solution
\cite{user}.  We solve this question when the initial condition
$\eta^{(L)}$ approximates a profile $\phi_0$ that, in addition to
satisfying Assumption \ref{assumption:1}, is convex.
%

%Namely, given  $\rho_i\in \stackrel\circ{\mathbb T}, i\le k\in \mathbb N$ we define
% \begin{eqnarray}
%   \label{eq:phi0+}
%   \phi_0(x)=\max_{i\le k}\rho_i x. %, \quad   \phi_0^-(x)=\min_{i\le k}\rho_i x.
% \end{eqnarray}
Let $\phi(x,t)$ % (resp. $\phi^-(x,t)$)
denote the viscosity solution of \eqref{eq:PDE} with initial condition
$\phi_0(\cdot)$% (resp. $\phi^-_0(\cdot)$)
. Since $\phi_0(\cdot)$ % (resp. $\phi_0^-(\cdot)$)
is convex% (resp. concave)
, the viscosity solution is given by the following ``Hopf formula''
\cite{BE}:
\begin{eqnarray}
  \label{eq:solvisc}
  \phi(x,t)=\sup_{y\in \mathbb R^2}\{y\cdot x-v(y)t-\phi_0^*(y)\}=[t v+\phi_0^*]^*(x),
\end{eqnarray}
where 
\begin{eqnarray}
  \label{eq:Leg}
  f^*(y)=\sup_{z\in \mathbb R^2}\{z\cdot y-f(z)\}
\end{eqnarray}
denotes  the Legendre-Fenchel transform of a function $f:\mathbb R^2\mapsto \mathbb R\cup\{+\infty\}$.

Let $\mathcal A$ denote the range of
the sub-differential of $\phi_0$: $\mathcal A$ is an almost-convex set (i.e. it contains the interior of its convex hull) 
\cite{Minty} and, in view of Assumption \ref{assumption:1}, its closure $\overline{\mathcal A}$ is a subset of
$ \stackrel\circ{\mathbb T}$.  Formula \eqref{eq:solvisc} makes sense even
though $v(\cdot)$ is not defined outside $\mathbb T$: in fact, note
that $\phi^*_0(y)=+\infty$ outside $\overline{\mathcal
  A}$ and the supremum in \eqref{eq:solvisc} can be restricted to $y\in  \stackrel\circ{\mathbb T} $. % Note that

\begin{Theorem}
\label{th:shock}
Let the initial condition $\eta^{(L)}$  satisfy Assumption \ref{assumption:1}
 and let in addition  $\phi_0:\mathbb R^2\mapsto \mathbb R$ be convex.
 For every $x\in\mathbb R^2, t>0$ and $\delta>0$ one has 
\eqref{eq:convergenza}, where $\phi(x,t)$ is given by \eqref{eq:solvisc}.
\end{Theorem}
\begin{Remark}
  The theorem holds also, with an identical proof, if $\phi_0(\cdot)$ is concave instead. Of course,  in this case the
  viscosity solution \eqref{eq:solvisc} has to be modified in the obvious way (both in \eqref{eq:solvisc} and in \eqref{eq:Leg} the $\sup$ becomes an $\inf$).
\end{Remark}

In general it is not possible to solve the variational principle
\eqref{eq:solvisc} explicitly. However, the following qualitative
result shows that in the generic case singularities of the gradient do
appear for sufficiently large  times:
\begin{Proposition}
  Let $\phi_0(\cdot)$ be convex and assume that $\mathcal A$ has an
  non-empty interior. Then, there exists $t_0<\infty$ such that, for every $t>t_0$, the function
  $\phi(\cdot,t):\mathbb R^2\mapsto \mathbb R$ is not everywhere  differentiable.
\label{th:guillaume}
\end{Proposition}

As a special example of convex initial condition, consider the case
% \eqref{eq:phi0+}, \eqref{eq:phi0-} is the case $k=2$ (or when $k\ge3$ but $\mathcal C(\rho_1,\dots,\rho_k)$ reduces to a segment, in which case only two slopes effectively enter the definition \eqref{eq:phi0+}). In this
% situation, one can rewrite the initial condition as 
\begin{eqnarray}
  \label{eq:Rie1}
  \phi_0(x)=c \,x\cdot \beta +\psi_0(x\cdot {n})
\end{eqnarray}
where ${ n},\beta\in\mathbb S^1$ are  unit vectors, $c\in\mathbb R$
and
\begin{eqnarray}
  \label{eq:Rie2}
  \psi_0(y)=\left\{
  \begin{array}{ll}
    u_- y,& y\le 0\\
u_+ y, & y\ge0
  \end{array}
\right.
\end{eqnarray}
with $u_-,u_+$ two  real constants.
Of course, we require that 
\begin{eqnarray}
  \label{eq:Rie3}
c\,\beta+u_\pm {n}\in \stackrel\circ{\mathbb T}.
\end{eqnarray}
In this case, the equation \eqref{eq:PDE} is effectively one-dimensional and the viscosity solution \eqref{eq:solvisc} 
 reduces to
\begin{eqnarray}
  \label{eq:Rie5}
  \phi(x,t)=c \,x\cdot \beta +\psi(x\cdot { n},t)
\end{eqnarray}
with $\psi(y,t)$ the unique viscosity solution of 
\begin{eqnarray}
  \label{eq:Rie6}
\left\{
  \begin{array}{l}
\partial_t\psi(y,t)+V(\partial_y\psi(y,t))=0    \\
\psi(y,0)=\psi_0(y),
  \end{array}
\right.
\end{eqnarray}
where 
\begin{eqnarray}
  \label{eq:VV}
  V(s)=v(c \beta+s\, {n}).
\end{eqnarray}
The space derivative $u(y,t):=\partial_y \psi(y,t)$ is the unique
entropy solution of the Riemann problem for the the following
one-dimensional scalar conservation law \cite{HH,Evans}:
\begin{eqnarray}
  \label{eq:Rie7}
\left\{
  \begin{array}{l}
\partial_tu(y,t)+\partial_y V(u(y,t))=0    \\
u(y,0)=u_-{\bf 1}_{y<0}+u_+{\bf 1}_{y>0}
  \end{array}
\right. \quad y\in \mathbb R, t>0.
\end{eqnarray}

From the theory of one-dimensional conservation laws \cite{HH}, we know that the regularity of
$\psi(\cdot,t)$ for $t>0$ depends on the convexity properties of the function
$V(\cdot)$.  Namely (assume to fix ideas that $u_-<u_+$, and an analogous picture holds in the opposite case):
\begin{enumerate}
\item [(a)] if the lower convex envelope $ V^{**}$ of the function $V:u\in[u_-,u_+]\mapsto V(u)$ is strictly
  convex in $[u_-,u_+]$ then $\psi(\cdot,t)$ is smooth for every
  $t>0$ and solves \eqref{eq:Rie6} in the classical sense (this
  corresponds to a ``rarefaction fan'' for \eqref{eq:Rie7}). 
\item [(b)] In instead $ V^{**}$ has one or several flat pieces, then
$\partial_y\psi(\cdot,t)$ has one or several discontinuities also for $t>0$,
corresponding in terms of \eqref{eq:Rie7} to travelling shocks.
\end{enumerate}
The fact that the function
$v:\rho\in \stackrel\circ{\mathbb T}\mapsto v(\rho)$ has everywhere a
Hessian with mixed signature shows that both cases (a) and (b) above
can occur. Actually, in view of Proposition \ref{th:guillaume} above
we see that the only case where the solution with convex initial
profile is differentiable in space for every $t>0$ is when
$\mathcal A$ is a segment and $v$ restricted to that segment is
convex.

Note also that, as soon as
$\mathcal A$ has an non-empty interior, the
solution \eqref{eq:solvisc} is genuinely two-dimensional, in contrast with \eqref{eq:Rie5}.

\section{Proof of Proposition \ref{prop:bendefinito}}
\label{sec:bendefinito}

\subsection{Some preliminary results}

We start with a few useful properties of the dynamics.

\begin{Proposition} 
\label{prop:aux1}
Let $t,W,\eta,(p,\ell)$ be as in Definition \ref{def:Xi}.
If a non-empty finite collection $\xi $ of points $((x,s),(p',\ell'))$ verifies  the
conditions (I)-(III) of Definition \ref{def:Xi}, then there exists $\emptyset\ne\xi'' \subset \xi$
such that $\xi''\in \Xi_{(p,\ell),\eta,W,t}$. %Since there is a unique
%point with particle label $(p,\ell)$ in $\xi$ and $\xi''$, w
%e
%obviously have
Moreover one has
 $x_0(\xi)=x_0(\xi'')$.
\end{Proposition}

\begin{Definition}
Given a configuration $\eta\in\Omega$, we say that $((x,s),(p,\ell))\stackrel \eta\mapsfrom((x',s'),(p',\ell'))$ if $(p,\ell)\in I_{(p',\ell')}$, $x'<z_{(p,\ell)}$ and
$s\le s',x\le x'$.
\end{Definition}

\begin{proof}
 We construct $\xi''$ starting from $\xi$ and %Let's take such a $\xi$ and construct $\xi'' $ by 
removing redundant points. We do this iteratively:
\begin{enumerate}
\item First we remove any point of $\xi$ that does not verify
  condition (V) and we stop when all remaining points verify it. Thus we obtain a new set $\xi^{(1)}\subset \xi$ that
  is not empty (because the right-most point has not been
  removed). Either $\xi^{(1)}\in \Xi_{(p,\ell),\eta,W,t}$ (in which
  case we set $\xi'':=\xi^{(1)}$) or it satisfies all conditions
  required to be in $ \Xi_{(p,\ell),\eta,W,t}$ except 
  condition (IV);
\item In the latter case, $\xi^{(1)}$ contains a point  $((x,s),(p',\ell'))$ as well as, say,
two distinct points  $((x_1,s_1),(q,m))$ and $((x_2,s_2),(q,m))$ such that
$((x_i,s_i),(q,m))\stackrel \eta\mapsfrom(x,s,p',\ell')$, $i=1,2$.
 %% for a particle index $(q,m)\in I_{(p',\ell')}$ such that $x< z_{(q,m)}$ and $x_1\le x, x_2\le x, s_1\le s, s_2\le s$.
Call $(q',m')$ the unique particle index, different from $(p',\ell')$,
such that $(q,m)\in I_{(q',m')}$. If there exists a point
$((x',s'),(q',m'))\in \xi^{(1)}$ such that $((x_1,s_1),(q,m))\stackrel
\eta\mapsfrom((x',s'),(q',m'))$ (resp.  $((x_2,s_2),(q,m))\stackrel
\eta\mapsfrom((x',s'),(q',m'))$) then we define
$\tilde\xi^{(1)}\subset \xi^{(1)}$ by removing $((x_2,s_2),(q,m))$
(resp. $((x_1,s_1),(q,m))$). Else, we define $\tilde\xi^{(1)}$ by
removing one of the two points $((x_i,s_i),(q,m))$ (it does not matter which one).

\item We define $\xi^{(2)}\subset \tilde \xi^{(1)}$ by removing all
  points that do not satisfy condition (V). By construction, as was
  the case for $\xi^{(1)}$ in point (1), we see that $\xi^{(2)}$ is
  non-empty and either it belongs to $\Xi_{(p,\ell),\eta,W,t}$ (in
  which case we set $\xi'':=\xi^{(2)}$) or it satisfies all conditions
  required to be in $ \Xi_{(p,\ell),\eta,W,t}$ except for condition (IV).
\item In the latter case, we iterate the procedure described above and
  obtain $\xi^{(j)}$, $j\ge 1$. The procedure stops after a finite
  number $n$ of iterations, since $\xi$ is finite and at each step one removes at least a point. By construction,
  $\xi^{(n)}$ is non-empty, it belongs to $\Xi_{(p,\ell),\eta,W,t}$ and $x_0(\xi^{(n)})=x_0(\xi).$
\end{enumerate}

\end{proof}      
  
\begin{Proposition} 
\label{prop:estratto} If $((x,s),(p',\ell'))\in \xi \in  \Xi_{(p,\ell),\eta,W,t} $, then there exists a non-empty $\xi'\subseteq \xi$ such that $\xi'\in  \Xi_{(p',\ell'),\eta,W,s}$ and the point with right-most space coordinate is $((x,s),(p',\ell'))$.
 \end{Proposition}

\begin{proof}
Again, we construct $\xi'$ iteratively: this time we start at step $0$ with
$\xi'_0=\emptyset$ and at each step we add a certain number of points taken
from $\xi$, until we stop after a finite number of steps. More
precisely:
\begin{itemize}
\item at step $1$ we add just one point and we set
  $\xi'_1=\{((x,s),(p',\ell'))\}$. Note that this set satisfies all
  properties in Definition \ref{def:Xi}, except possibly property
  (III);
\item for $k\ge1$, either the set $\xi'_k$ satisfies all conditions in Definition \ref{def:Xi} (in which case we set $\xi':=\xi'_k$), or some of the  points added at step $k$ do not satisfy condition (III);
\item in the latter case, let $((x_j^{(k)},s_j^{(k)}),(p_j^{(k)},\ell_j^{(k)})), j\le n^{(k)}$ denote the points that do not satisfy condition (III). One proceeds as follows: for each $j\le n^{(k)}$ and for each $(q,m)\in I_{(p^{(k)}_j,\ell^{(k)}_j)}$
such that $x_j^{(k)}<z_{(q,m)}$, we add the unique point $((x',s'),(q,m))\in \xi$ such that $x'\le x^{(k)}_j,s'\le s^{(k)}_j$. The union of $\xi'_k$ with the added points defines $\xi'_{k+1}$.
\end{itemize}

Since $\xi$ is finite, the procedure stops after a finite number
$\bar k$ of steps; by construction $\xi':=\xi'_{\bar k}$ satisfies all
conditions to be a non-empty element of $\Xi_{(p',\ell'),\eta,W,s}$,
the point $((x,s),(p',\ell'))$ belongs to it and all other points have
space coordinate smaller than $x$ and time coordinate smaller than
$s$.  Note that by construction all points in $\xi'$ satisfy property
(V), because the points $((x',s'),(q,m))$ we add along the procedure
are either the rightmost point $((x,s),(p',\ell'))$ or are such that
there exists already another point $((x'',s''),(q',m'))$ with
$ ((x,s),(p',\ell'))\stackrel \eta\mapsfrom
((x'',s''),(q',m'))$.
Also, property (IV) is automatically satisfied since
$\xi'\subset \xi$.
\end{proof} 
 
\begin{Definition} \label{def:chain} Given $\xi \in
  \Xi_{(p,\ell),\eta,W,t}$ we define a \emph{decreasing path of size
    $k\ge0$ } in $\xi$ to be a sequence
  $((x_1,s_1),(p_1,\ell_1)),...,((x_k,s_k),(p_k,\ell_k))$ of points in
  $\xi$ such that:
\begin{itemize}
\item $(x_i)_{1 \leq i \leq k}$ and  $(s_i)_{1 \leq i \leq k}$ are strictly decreasing.
\item For every $ i>1$, $(p_i,\ell_i)\in I_{(p_{i-1},\ell_{i-1})}$.
\end{itemize}
We  let  $diam(\xi)$ (the \emph{diameter of $\xi$}) be the maximal size of a  decreasing path in $\xi$.

\end{Definition} 
 
Given $\xi\in \Xi_{(p,\ell),\eta,W,t}$ and a subset $R\subset \mathbb R\times \mathbb R^+$, we will say with some abuse of notation that $\xi\subset R$ if the space-time coordinates of each point of $\xi$ belong to $R$.

\begin{Proposition}\label{prop:diametro}
For every $a\in \mathbb R, h\ge0, t\ge0, k\in\mathbb N,\eta\in \Omega$ and particle index $(p,\ell)$ we have:
\begin{eqnarray}
\label{eq:catena}\mathbb P(\exists \xi \in \Xi_{(p,\ell),\eta,W,t}|\xi \subset [a,a+h)\times[0,t], diam(\xi)\ge k)\leq \frac{(4th)^k}{(k!)^2}.
\end{eqnarray}
\end{Proposition} 
\begin{proof}
We remark first of all that, given $\xi\in \Xi_{(p,\ell),\eta,W,t}$, a
decreasing path of maximal size is necessarily such that
$(p_1,\ell_1)=(p,\ell)$ (otherwise, by property (V) in Definition
\ref{def:Xi}, we could construct a longer decreasing path). Therefore,
for the event in \eqref{eq:catena} to happen, we need a chain of points
$((x_1,s_1),(p_1,\ell_1)),...,((x_k,s_k),(p_k,\ell_k))$ as in Definition
\ref{def:chain}, with $(x_i,s_i)\in [a,a+h)\times[0,t]$ and $(p_1,\ell_1)=(p,\ell)$. Let us fix
the chain $C=(p_1,\ell_1)\dots,(p_k,\ell_k)$. Later we will sum over the $2^k$ possible choices of $C$. 

There are at most $2h$ different half-integer points (possible values
of $x_i$) in the interval $[a,a+h)$. Therefore, there are at most
$\binom{2h}{k}$ choices for the strictly increasing sequence
$(x_i)_{i\le k}$ (this is an overcount, since the $x_i$ alternate
between $\mathbb Z$ and $\mathbb Z+1/2$). Once the chain $C$ and the
sequence $(x_i)_{i\le k}$ are fixed, the event that there exists the
decreasing path $((x_i,s_i),(p_i,\ell_i))$ is the event that there is
a sequence of temporally decreasingly ordered rings of the Poisson
clocks at $x_i$. This event has the same probability as the event that
a single rate-$1$ Poisson clock rings at least $k$ times before time
$t$.

In formulas, 
 \begin{align}
\mathbb P(\exists \xi \in \Xi_{(p,\ell),\eta,W,t}|\xi \subset [a,a+h)\times[0,t], diam(\xi)\ge k)
\leq \sum_C \binom{2h}{k} \sum_{i\geq k} \frac{t^i}{i!}e^{-t}\\
\leq 2^ k\binom{2h}{k} \frac{t^k}{k!}\sum_{i\geq 0} \frac{t^i}{i!}e^{-t}
=2^k\frac{t^k}{(k!)^2}\frac{(2h)!}{(2h-k)!}
\leq 2^k\frac{(2ht)^k}{(k!)^2}
\end{align} 
which proves the claim.
\end{proof}

\subsection{Proof of Proposition \ref{prop:bendefinito}}

\emph{Proof of Claim (1)}.
Let $\epsilon$ be a positive real number such that $0<\epsilon \leq \frac{e^{-3}}{4t}$. We have, from Proposition \ref{prop:diametro}  and Stirling's formula,
 \begin{eqnarray}
\label{eq:BC}\mathbb P(\exists \xi \in \Xi_{(p,\ell),\eta,W,t}|\xi \subset [z_{(p-k,\ell)},z_{(p-k,\ell)}+\epsilon k^2)\times[0,t], diam(\xi)\ge k)\leq C e^{-k}
 \end{eqnarray}
for some absolute constant $C$.
By Borel-Cantelli, there exists an almost surely finite random
variable $k_0$ such that 
\begin{eqnarray}
\label{k0}
\forall k\ge k_0, \nexists \xi \in \Xi_{(p,\ell),\eta,W,t}:\;\xi \subset [z_{(p-k,\ell)},z_{(p-k,\ell)}+\epsilon k^2)\times[0,t], diam(\xi)\ge k.
\end{eqnarray}
 In
addition to this, from the definition of $ \Omega$, there exists
a finite $k_1$ (depending on $\eta$, $(p,\ell)$ and $\epsilon$) such that
\begin{eqnarray}
\label{k1}
 z_{(p-k,\ell)}+\epsilon k^2 >
z_{(p,\ell)}\;\text{ for every } k\ge k_1.
\end{eqnarray}
Define the ``left-most point of $\xi$'' to be the point in $\xi$ with
left-most spatial coordinate. We need the following (see later for the
proof):
\begin{Lemma} \label{lemma:astuto}If the left-most point of
  $\xi\in \Xi_{(p,\ell),\eta,W,t}$ lies (weakly) to the left
  of % between $z_{(p-k-1,\ell)}$ and
  $z_{(p-k,\ell)}$ for some $k\ge 0$, then ${\rm diam}(\xi)\ge
  k+1$.
  If instead the left-most point of $\xi\in \Xi_{(p,\ell),\eta,W,t}$
  lies (weakly) to the right of $z_{(p,\ell)}-n$ and
  $x_0(\xi)\le z_{(p,\ell)}$, then ${\rm diam}(\xi)\le 2n$.
\end{Lemma}

Next, we claim that any $\xi$ whose left-most point is to the left of $z_{(p-k,\ell)}$, with $k>\max(k_0,k_1)$, cannot realize the infimum \eqref{eq:3}.
Assume that $k>\max(k_0,k_1)$ and that the space coordinate of the
left-most point of $\xi$ is in $(z_{(p-k-1,\ell)},z_{(p-k,\ell)}]$. Since the diameter of $\xi$ is at least 
$k+1\ge k_0$, by \eqref{k0} the right-most point of $\xi$ has horizontal
coordinate at least $z_{(p-k-1,\ell)}+\epsilon (k+1)^2 >z_{(p,\ell)}$,
see \eqref{k1}. Therefore, every $ \xi\in\Xi_{(p,\ell),\eta,W,t}$ with 
left-most point to the left of $z_{(p-k,\ell)}$, with
$k>\max(k_0,k_1)$, is irrelevant when it comes to the infimum \eqref{eq:3}.

 The only sets $\xi$ of interest for \eqref{eq:3} thus have points
 whose space-time coordinates lie in a finite rectangle
 $R=[a,z_{(p,\ell)}]\times [0,t]$. Such $\xi$ have diameter at most
 $D=2(z_{(p,\ell)}-a)$ (the factor $2$ is because in the definition of decreasing path one has $x_i-x_{i-1}\ge 1/2$) and therefore have particle with line labels
 $\ell $ in the finite set $\mathcal I=\{\ell-D,\dots,\ell+D\}$.
 Since the collection of Poisson processes $\mathcal W_{(\ell,z)},
 \ell\in \mathcal I, z\in [a,z_{(p,\ell)}]$ has only an almost-surely finite
 number of points in the time interval $[0,t]$, the infimum \eqref{eq:3} involves
 (almost surely) only a finite number of $\xi$, and is therefore a
 minimum.

\medskip
\emph{Proof of Claim (2)}. First of all, let us check that $z_{(p,\ell)}(t)\in \mathbb Z+(\ell\!\!\mod 2)/2$, as it should.
This is simply because $x_0(\xi)$ in \eqref{eq:3} belongs to $\mathbb Z+(\ell\!\!\mod 2)/2$, as follows from properties (I)-(II) of Definition 
\ref{def:Xi}. 

Next, we need to prove that the dynamics preserves the order between
particles. To do that, it is enough to prove that
$z_{(p,\ell)}(t)<z_{(p,\ell+1)}(t) $ and
$z_{(p,\ell)}(t)<z_{(p+1,\ell-1)}(t)$. We prove the former inequality
and for the latter an analogous argument works. If
$x_0(\xi)>z_{(p,\ell)}$ for every $\xi\in \Xi_{(p,\ell+1),\eta,W,t}$
then $z_{(p,\ell+1)}(t)>z_{(p,\ell)}\ge z_{(p,\ell)}(t)$ as wished.
If instead there exists $\xi\in \Xi_{(p,\ell+1),\eta,W,t}$ such that
$x_0(\xi)<z_{(p,\ell)}$ then by point (III) in Definition \ref{def:Xi}
there exists $((x,s),(p,\ell))\in \xi$ with $x< x_0(\xi)$. By
Proposition \ref{prop:estratto}, there exists
$\xi'\in \Xi_{(p,\ell),\eta,W,t}$ with $x_0(\xi')=x<x_0(\xi)$.
%%   Define $\xi'\subset \xi$ by taking $((x,s),(p,\ell+1))$ and then iteratively, for each point $\xi'=((x',s'),p',\ell')$ already in $\xi'$, adding the points $\xi''=((x'',s''),p'',\ell'')\in\xi$ such that $(p'',\ell'')\in \xi$. This implies that there
%% exists $ \xi'\in \Xi_{(p,\ell+1),\eta,W,t}$ with $x_0(\xi')<x_0(\xi)$
%% ($\xi'\subset w$ can be constructed by taking $((x,s),(p,\ell+1))$ and
 Taking the infimum over $\xi\in \Xi_{(p,\ell+1),\eta,W,t}$ allows to
conclude that $z_{(p,\ell)}(t)<z_{(p,\ell+1)}(t) $.

Finally, we need to verify that 
\[
\lim_{p\to-\infty}\frac{z_{(p,\ell)}(t)}{p^2}=0.
\]
This is similar to the proof of Claim (1). Let $0<\epsilon \leq \frac{e^{-3}}{4t}$. We have, from \eqref{eq:BC} with $p=-k$,
 \begin{eqnarray}
\mathbb P(\exists \xi \in \Xi_{(-k,\ell),\eta,W,t}|\xi \subset [z_{(-2k,\ell)},z_{(-2k,\ell)}+\epsilon k^2)\times[0,t], diam(\xi)\ge k)\leq C e^{-k}.
 \end{eqnarray}
By Borel-Cantelli, there exists an almost surely finite random
variable $k_0$ such that 
\begin{eqnarray}
\forall k\ge k_0, \nexists \xi \in \Xi_{(-k,\ell),\eta,W,t}:\;\xi \subset [z_{(-2k,\ell)},z_{(-2k,\ell)}+\epsilon k^2)\times[0,t], diam(\xi)\ge k.
\end{eqnarray}
 Again, from the definition of $ \Omega$, there exists
a finite $k_1$ such that
%% \begin{eqnarray}
%%  z_{(-2k,\ell)}+\epsilon k^2 \geq
%% z_{(-k,\ell)}\;\forall k\ge k_1.
%% \end{eqnarray}
%%  Finally, from the definition of $\Omega$, there exists $k_2$ such that
\begin{eqnarray} z_{(-2k,\ell)} \geq - \epsilon k^2\text{ for every }
  k \ge k_1. \end{eqnarray} By choosing $k \ge \max(k_0, k_1)$ we make
sure that, as seen in the proof of Claim (1), the leftmost point of
any $\xi$ in $\Xi_{(-k,\ell),\eta,W,t}$ is to the right of
$z_{(-2k, \ell)}$. As a consequence,
\begin{eqnarray}  z_{(-k,\ell)}(t) \geq z_{(-2k, \ell)} \geq -\epsilon k^2 \end{eqnarray}
and
\[
0\ge \liminf_{p\to-\infty}\frac{z_{(p,\ell)}(t)}{p^2}\ge -\epsilon.
\]
By taking $\epsilon$ arbitrarily small   we get the result.     

\emph{Proof of Claim (3)}.  First, let us prove that
\[z_{(p,\ell)}(t) \geq \inf_{\xi \in
    \Xi_{(p,\ell),\eta(s),\theta_s(W),t-s}} x_0(\xi). \] Let us take
$\xi$ $\in$ $\Xi_{(p,\ell),\eta,W,t}$ for which the infimum
\eqref{eq:3} is reached, and consider its restriction $\xi'$ to
$[s,t]$, i.e. the subset of points of $\xi$ with time coordinate in
$[s,t]$. We will prove that from $\xi'$ we can construct a path in
$\xi''\in\Xi_{(p,\ell),\eta(s),\theta_s(W),t-s}$ with
$x_0(\xi)=x_0(\xi')=x_0(\xi'')$.  Then we deduce
\begin{equation}
z_{(p,\ell)}(t)=x_0(\xi'')\ge \inf_{\xi \in \Xi_{(p,\ell),\eta(s),\theta_s(W),t-s}}
x_0(\xi)
\end{equation}
as desired. 

In order to construct $\xi''$, we start by shifting temporally $\xi'$
by $-s$. We will verify that it satisfies conditions (I)-(III) from
the Definition \ref{def:Xi} of $\Xi_{(p,\ell),\eta(s),\theta_s(W),t-s}$:
then, the existence of $\xi''$ follows from Proposition
\ref{prop:aux1}.

Conditions (I)-(II) are obvious (recall also Remark \ref{rem:II}), so
we concentrate on (III). For $((x_1,s_1),(p_1,\ell_1))\in \xi'$,
suppose there exists $(p_2,\ell_2) \in I_{(p_1,\ell_1)}$ such that
$ x_1 < z_{(p_2,\ell_2)}(s)$. We also have
$ x_1 \leq z_{(p_2,\ell_2)}(0)$, since
$z_{(p_2,\ell_2)}(s)\le z_{(p_2,\ell_2)}(0)$.
This implies that there exists a point $((x_2,s_2),(p_2,\ell_2))$ in
 $\xi$ such that $ x_2 < x_1$ and $s_2<s_1$.  Two cases can in principle occur:
either $s_2\geq s$, and then $((x_2,s_2),(p_2,\ell_2))\in\xi'$, so
condition (III) is satisfied for point $((x_1,s_1),(p_1,\ell_1))$, or
$s_2<s$.  The latter case is however not possible: indeed, by
Proposition \ref{prop:estratto} we deduce that
$z_{(p_2,\ell_2)}(s)\leq x_2 < x_1$, which contradicts the assumption
$ x_1 < z_{(p_2,\ell_2)}(s)$ we made above.

 \smallskip

 Secondly, let us prove that
 \[z_{(p,\ell)}(t) \leq \inf_{\xi\in
     \Xi_{(p,\ell),\eta(s),\theta_s(W),t-s}} x_0(\xi).  \] To do this,
 we will show that for every $\xi$ in
 $\Xi_{(p,\ell),\eta(s),\theta_s(W),t-s}$ there exists $\xi'$ in
 $\Xi_{(p,\ell),\eta,W,t}$ such that $x_0(\xi)=x_0(\xi')$.  We denote
 $\xi''_{(0)}$ the temporal shift of $\xi$ by $+s$, which already
 verifies conditions (I)-(II) in Definition \ref{def:Xi} to belong to
 $\Xi_{(p,\ell),\eta,W,t}$. If condition (III) is also verified, then the existence of $\xi'$ follows
 from Proposition \ref{prop:aux1}. Otherwise, we modify iteratively $\xi''_{(0)}$ into a configuration $\xi''_{(k)},k\ge1$
 until it satisfies property (III), while at the same time keeping
 $x_0(\xi''_{(k)})$ constant and equal to $x_0(\xi)$, and then we conclude
 again with Proposition \ref{prop:aux1}. More precisely we proceed as
 follows.  If $((x_1,s_1),(p_1,\ell_1))\in \xi''_{(0)}$ and there
 exists $(p_2,\ell_2)\in I_{(p_1,\ell_1)}$ such that
 $ x_1 < z_{(p_2,\ell_2)}(0)$, then two alternatives
 arise: 
\begin{itemize}
\item either $ x_1 < z_{(p_2,\ell_2)}(s)$, and then $\xi''_{(0)}$ already
  contains a (unique) point $((x_2,s_2),(p_2,\ell_2))$ with $s\le s_2\le s_1,x_2\le x_1$,
which guarantees    condition (III);
\item or $  x_1    > z_{(p_2,\ell_2)}(s)$, in which case $z_{(p_2,\ell_2)}(s)\neq z_{(p_2,\ell_2)}(0)$. This means that  $\Xi_{(p_2,\ell_2),\eta,W,s} $ contains
a (non-empty) element $\tilde \xi$ with $x_0(\tilde \xi)=z_{(p_2,\ell_2)}(s)$.
Then, we define $\xi''_{(1)}:=\xi''_{(0)}\cup\tilde\xi$. This way, the point $((x_1,s_1),(p_1,\ell_1))$ now satisfies condition (III), by construction.
\end{itemize}

We iterate this procedure as long as there are points in $\xi''_{(j)}$ that
do not verify condition (III), thereby obtaining a sequence
$\xi''_{(j)},j\ge1$. The iteration necessarily stops after a
finite number $J\le |\xi''_{(0)}|$ of steps, since only the points in $\xi''_{(0)}$ may
possibly not satisfy condition (III) along the procedure.  Now $\xi''_{(J)}$ satisfies
conditions (I)-(III)  and we conclude the existence of $\xi'\in \Xi_{(p,\ell),\eta,W,t}$ 
with $x_0(\xi)=x_0(\xi')$ applying Proposition \ref{prop:aux1}.
\smallskip

\emph{Proof of Claim (4).}  Given $\xi\in \Xi_{(p,\ell),\eta,W,t}$
with $x_0(\xi)<z_{(p,\ell)}$ and ${\rm diam}(\xi)\ge n\ge 2$, let
$x_-(\xi)$ be the space coordinate of the left-most point in
$\xi$. Recall from Lemma \ref{lemma:astuto} that, if
$x_-(\xi)\in [z_{(p-k-1,\ell)},z_{(p-k,\ell)})$ then
${\rm diam}(\xi)>k$. Then,
\begin{multline}
\label{scomposiz}
\mathbb P(A_{(p,\ell)}(n)):=\mathbb P(\exists \xi\in
\Xi_{(p,\ell),\eta,W,t}:x_0(\xi)<z_{(p,\ell)}, {\rm diam}(\xi)\ge n)
\\
=\sum_{k\ge0}\mathbb P(A_{(p,\ell)}(n);x_-(\xi)\in
[z_{(p-k-1,\ell)},z_{(p-k,\ell)}))
\\
\le \sum _{k\ge0}\mathbb P(\exists \xi\in
\Xi_{(p,\ell),\eta,W,t}:\xi\in [z_{(p-k-1,\ell)},z_{(p,\ell)})\times
[0,t],{\rm diam}(\xi)\ge\max(n,k)).
\end{multline}
According to Proposition \ref{prop:diametro}, the latter sum is bounded by 
\begin{equation}
\label{eq:cia}
 \sum _{k\ge0}\frac{(4t (z_{(p,\ell)}-z_{(p-k-1,\ell)}))^{n\vee k}}{((n\vee k)!)^2}.
\end{equation}
By the assumption $\eta\in \Omega$, there exists $k_0=k_0(\eta,p,\ell)$ such that
\[
z_{(p,\ell)}-z_{(p-k-1,\ell)}\le (e^{-3}/4) k^2\quad \text{ for every } k>k_0
\]
which, together with \eqref{eq:cia}, implies that 
\[
\mathbb P(A_{(p,\ell)}(n))\stackrel{t\to0}=O(t^n).
\]
In particular, taking $n=2$, the claim follows (since $|\xi|\ge2$ implies ${\rm diam}(\xi)\ge2$).
\qed

\begin{proof}[Proof of Lemma \ref{lemma:astuto}]
  Call $((x,s),(p',\ell'))$ the left-most point of
  $\xi\in \Xi_{(p,\ell),\eta,W,t}$. By property (V) in Definition
  \ref{def:Xi} there exists in $\xi$ a decreasing path with particle
  indices $(p_d,\ell_d),\dots,(p_0,\ell_0)$ with
  $(p_d,\ell_d)=(p',\ell')$, $(p_0,\ell_0)=(p,\ell)$ and
  $(p_j,\ell_j)\in I_{(p_{j-1},\ell_{j-1})}$. Since
  $I_{(p,\ell)}=\{(p-1,\ell+1),(p,\ell-1)\}$, one sees that
  $(p',\ell')$ is of the form $(p-(i+d)/2,\ell+i)$ for some $d\ge0$
  and $|i|\le d$. Of course, the diameter of $\xi$ is at least $d+1$.

%% One sees that $2(p_{j-1}-p_j)+(\ell_{j-1}-\ell_j)=1$ always, which implies $r=2(i+d)/2-i=d$, so that the diameter of $\xi$ is at least $k$.

%%  has particle index of the form
%%  $(p',\ell')=and
The horizontal coordinate $x$ is by assumption smaller than $ z_{(p-k,\ell)}$. Assume for definiteness
 that $i\ge0$, the reasoning being analogous for $i<0$. Since $z_{(p,\ell)}< z_{(p,\ell+1)}$ for every $(p,\ell)$ by the particle interlacing constraints, we
 see that $x\le z_{(p-k,\ell+i)}$. We claim that $(i+d)/2\ge k$;
 in view of $i\le d$, this implies $d\ge k$ and therefore the desired claim that the diameter of $\xi$ is at least $k+1$. To see that $(i+d)/2\ge k$,
 observe that if we had $(i+d)/2\le k-1$ and $x\le z_{(p-k,\ell+i)}$
 we would have also $x\le z_{(p-(i+d)/2-1,\ell+i)}$ and therefore
 $x< z_{(p-(i+d)/2-1,\ell+i+1)}$. Since $(p-(i+d)/2-1,\ell+i+1)\in
 I_{(p-(i+d)/2,\ell+i)}=I_{(p',\ell')}$, by property (III) in Definition \ref{def:Xi}
 we conclude that $((x,s),(p',\ell'))$ is not the left-most point of
 $\xi$, which is a contradiction. 

To prove the second statement of the Lemma, just observe that in the Definition \ref{def:chain} of decreasing path the horizontal coordinates $x_i$ are strictly decreasing and actually $x_i-x_{i+1}\ge 1/2$.
\end{proof}

\section{Stationary measures, stochastic domination, propagation of information}

\label{sec:importante}
\subsection{Stationary  measures}
\label{sec:stationary}
As proved in \cite{Toninelli2+1}, for any $\rho=(\rho_1,\rho_2)\in \stackrel\circ{\mathbb T}$
the dynamics admits a stationary measure $\pi_\rho$ on $\Omega$, such
that
\begin{itemize}
\item $\pi_\rho$ is translation invariant and ergodic w.r.t. translations
\item the height has average slope $\rho$ under $\pi_\rho$, i.e. for any $x\in G^*$
\[
\pi_\rho(h_\eta(x+e_i)-h_\eta(x))=\rho_i,\;i=1,2, \;e_1=(1,0), e_2=(0,1);
\]
\item $\pi_\rho$ is locally uniform: given any finite subset $A\subset G^*$ and $\bar \eta\in \Omega$, the law $\pi_\rho$ conditioned to the (finite) set 
$E_{\bar \eta}:=\{h_\eta(x)= h_{\bar\eta}(x)$ for every $x\not \in A\}$ is the uniform distribution on $E_{\bar\eta}$.
\end{itemize}
\begin{Remark}
  To be precise, in \cite{Toninelli2+1} the dynamics was defined
  through a different procedure that does not involve  the
  graphical construction explained above nor the variational
  characterization \eqref{eq:3}. More precisely, in
  \cite{Toninelli2+1} one takes $K>0$ and first defines the dynamics
  with ``cut-off'' $K$, where Poisson clocks at positions $(\ell,z)$
  with $\|(\ell,z)\|\ge K$ are ignored (i.e. their rate is set to
  zero). In this case, the process is effectively a Markov chain on a
  finite state space and there is no difficulty in defining the particle
  positions $z^{(K)}_{(p,\ell)}(t)$. Next, one defines
  $z_{(p,\ell)}(t)$ by sending the $K$ to infinity: the limit exists
  because $z^{(K)}_{(p,\ell)}(t)$ is decreasing in $K$.  Such
  definition of $z_{(p,\ell)}(t)$ actually coincides almost surely with
  \eqref{eq:3}: this is because, as we have seen in the proof of Claim
  (1) of Proposition \ref{prop:bendefinito}, there exists an almost-surely finite random variable $K$ such the r.h.s. of 
\eqref{eq:3} is unchanged if the Poisson points with $\|(\ell,z)\|\ge K$ are removed from   $W$.
The graphical construction of the present work is definitely more convenient
for our subsequent goal of proving a hydrodynamic limit.
\end{Remark}

It was proven in \cite[Theorem 3.1]{Toninelli2+1} that in the stationary state
$\pi_\rho$ the interface moves with non-zero average speed $v(\rho)$, i.e. 
\[
\mathbb E_{\pi_\rho}(H(x,t)-H(x,0))=-t v(\rho)
\]
for any $x\in G^*$, with $\mathbb P_{\pi_\rho}$ the law of the process
started from $\pi_\rho$. It was subsequently proven in \cite{CF} that
actually $v(\cdot)$ is the same function as in \eqref{eq:v}.  It was
also proven in \cite[Theorem 3.1]{Toninelli2+1} that fluctuations of
$H(x,t)-H(x,0)$ in the stationary process grow slower than any power
of $t$: for any given $x\in G^*$ and $\eta>0$,
\begin{eqnarray}
  \label{eq:flutt}
\lim_{t\to\infty}  \mathbb P_{\pi_\rho}(|H(x,t)-H(x,0)+v(\rho)t|\ge t^\eta)=0.
\end{eqnarray}

The measures $\pi_\rho$ are in a sense very explicit: as discussed in
\cite{Toninelli2+1}, they are the infinite-volume Gibbs measures on
lozenge tilings of the plane, with prescribed densities
$\rho_1,\rho_2,\rho_3:=1-\rho_1-\rho_2$ for the three types of
lozenges \cite{KOS}. Such measures have a determinantal
representation: $n$-point correlation functions can be expressed as a
determinant of a $n\times n$ matrix, whose elements involve the
inverse of the so-called Kasteleyn matrix of the infinite hexagonal
lattice. In the present work, we will not make use of such
determinantal structure, nor of the fact that height fluctuations
under $\pi_\rho$ tend on large scales to a massless Gaussian field. We
will however need a couple of rougher estimates on the probability of
large height fluctuations.

A first fluctuation estimate we will need is the following:
\begin{Lemma}\cite[Prop. 5.7]{Zero}
  \label{lemma:GFF}
For every $\rho\in \stackrel\circ{\mathbb T}$ and $\epsilon>0$, there exists $c>0$ so that
\begin{eqnarray}
  \pi_\rho(\exists x\in G^*, |x|\le L: |h_\eta(x)-h_\eta(0)-\rho\cdot x|\ge (\log L)^{1+\epsilon})\le
\frac1c e^{-c(\log L)^{1+\epsilon}}.
\end{eqnarray}
\end{Lemma}

Also, we recall:
\begin{Lemma}
\label{lemma:A1}
\cite[Lemma A.1]{Toninelli2+1}
Let $I_r$ be the subset of points of the line labelled, say, $\ell=0$, having  horizontal coordinate in $[1,\dots,r]$. 
Given $\eta\in \Omega$, let $N_r=N_r(\eta)$ be the number of particles on $I_r$.

 For any $\lambda>0$, $u>0$ and $\rho\in \stackrel\circ{\mathbb T}$ 
there exists $C=C(\lambda,u,\rho)<\infty$ such that, for every $r\in \mathbb N$,
  \begin{eqnarray}
    \label{eq:Nr}
    \pi_\rho(|N_r-r\rho_3|\ge u r)\le C e^{-\lambda u r}.
  \end{eqnarray}
\end{Lemma}
In particular, the probability of large particle spacings decays faster than exponential.
 A simple consequence is the following:
 \begin{Corollary}
   \label{cor:A1}
   For any  $\rho\in \stackrel\circ{\mathbb T} $  and $M>1/\rho_3$ there
   exists $C=C(\rho)<\infty$ such that,
   for every particle index $p$ on row $\ell=0$ and every
   $k\in\mathbb N$
 \begin{eqnarray}
   \label{eq:corA1}
   \pi_\rho(z_{(p,0)}-z_{(p-k,0)}> M k)\le C (p\vee k)e^{-k}.
 \end{eqnarray}
 \end{Corollary}
 \begin{proof}[Proof of Corollary \ref{cor:A1}]
   We have
   \begin{multline}
  \pi_\rho(z_{(p,0)}-z_{(p-k,0)}> M k)\\\le \pi_\rho( z_{(p,0)}\ge M (k\vee p))+\pi_\rho[z_{(p,0)}-z_{(p-k,0)}> M k;  z_{(p,0)}< M (k\vee p)].
   \end{multline}
   On one hand,  recall that the left-most particle on line $0$ with non-negative coordinate is labelled $(0,0)$: then,
$z_{(p,0)}\ge M (k\vee p)$ implies that
   $N_{M(k\vee p)}\le p$ and by Lemma \ref{lemma:A1} this event has
   probability at most $C e^{-k\vee p}$ if $M$ is strictly larger than
   $1/\rho_3$. On the other hand, if $z_{(p,0)}< M (k\vee p)$ and
   $z_{(p,0)}-z_{(p-k,0)}> M k$ then there exists a translation of
   $N_{Mk}$ by $0\le j\le M(k\vee p)$ which contains at most $k$
   particles. Applying again Lemma \ref{lemma:A1} this has probability
   at most $M(k\vee p)C e^{-k}$, where the prefactor comes from the
   union bound on $j$.
 \end{proof}
 \begin{Remark}
\label{rem:bernnonbern}
In the stationary measure $\pi_\rho$ the particle density is $\rho_3$
but, as is clear from the fact that in  Lemma \ref{lemma:A1} one can take $\lambda$ as large as wished, on each line $\ell$ the
particle process is much more ``rigid'' than a Bernoulli
i.i.d. process of density $\rho_3$. In contrast, let us recall that
the translation-invariant stationary measures of the one-dimensional
Hammersley-Aldous-Diaconis process are i.i.d. Bernoulli \cite{FerrariMartin,SeppaH}.
 \end{Remark}

\subsection{Stochastic domination}
\label{sec:stochdom}

Recall that, as in Definition \ref{def:omega}, we fix particle labels
so that the left-most particle on line $\ell=0$, with non-negative
horizontal coordinate, is labeled $(p,\ell)=(0,0)$.  Recall also, from 
Definitions \ref{def:coord} and \ref{def:height}, that the height function is defined at
vertices $x=(x_1,x_2)$ of $G^*$ and that vertex $x$ is on line $\bar\ell(x)$ and has horizontal coordinate $\bar z(x)$.

\begin{Lemma} Take two configurations $\eta,\eta'\in \Omega$.  For $t\ge0$ and
  $x\in G^*$ let $p(x,t)\in\mathbb Z$ be the unique index such that
  $z_{(p(x,t)-1,\bar \ell(x))}<\bar z(x)<z_{(p(x,t),\bar\ell(x))}$ and
  similarly for $p'(x,t)$. Denoting $H(\cdot,t),H'(\cdot,t)$ the height functions at time $t$ with initial conditions $\eta,\eta'$, we have
\begin{eqnarray}
  \label{eq:aaa}
  H(x,t)-H'(x,t)=p'(x,t)-p(x,t)+h_{\eta}(0)-h_{\eta'}(0)
\end{eqnarray}
where $h_\eta(0):=h_\eta(x)|_{x=(0,0)}$. \end{Lemma}
\begin{proof}
  First let us prove the formula for $t=0$ and $\bar \ell(x) =0$. From the definition of the gradient of the height function
  (and more precisely from \eqref{eq:grad2}),
\[ \begin{aligned} H(x,0)=h_{\eta}(0)+\bar z(x)-p(x,0) +1/2\\
  H'(x,0)=h_{\eta'}(0)+\bar z(x)-p'(x,0) +1/2
\end{aligned} \]
so that \eqref{eq:aaa} holds for $t=0$ and $\bar \ell(x)=0$.  Next we 
 prove \eqref{eq:aaa}  for $t=0$ and $\bar
\ell(x)\ne0$.
Let us proceed by induction: suppose the claim is true for $\bar
\ell(x)=k$ and we want to prove it for $\bar\ell(x)=k+1$ (the induction from $k$ to $k-1$ works the same way).
If $\bar\ell(x)=k+1$, let $y=x-(0,1)$ (recall that $y$ is on line $k$ if $x$ is on line $k+1$).  The following  cases can arise (see Figure \ref{fig:5}): 
\begin{figure}
\includegraphics[width=12cm]{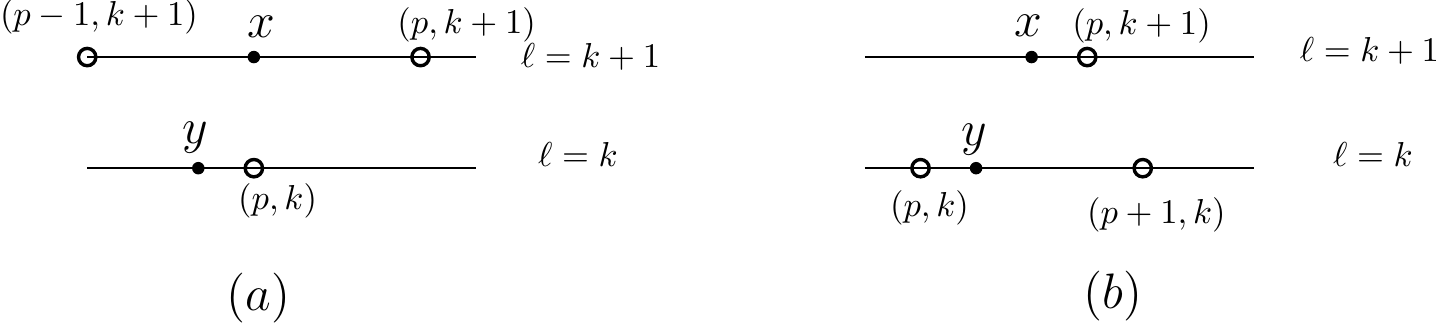}
  \caption{In drawing (a) a particle configuration such that $p(y,0)=p(x,0)$ while in drawing (b) a configuration where $p(y,0)=p(x,0)+1$.
 }
\label{fig:5}
\end{figure}

\begin{itemize}
\item either
  \begin{eqnarray}
    \label{eq:caso1}
 p(x,0)=p(y,0)\quad\text{ and }\quad  p'(x,0)=p'(y,0)  .
  \end{eqnarray}
  We have then
  $ H(y,0)-H'(y,0)=p'(x,0)-p(x,0)+h_{\eta}(0)-h_{\eta'}(0)$ by
  the induction hypothesis. On the other hand, from \eqref{eq:height},
  we see that \eqref{eq:caso1} implies
  $H(y,0)=H(x,0)$, $H'(y,0)=H'(x,0)$. The claim follows.
\item or
  \begin{eqnarray}
    \label{eq:caso2}
    p(y,0)=p(x,0)+1\quad\text{ and }\quad p'(y,0)=p'(x,0)+1.
  \end{eqnarray}
Again by induction we deduce
  $ H(y,0)-H'(y,0)=p(x,0)-p'(x,0)+h_{\eta}(0)-h_{\eta'}(0)$. In this case, \eqref{eq:height} implies that both $H(\cdot,0)$ and $H'(\cdot,0)$ increase by $1$ when going from $y$ to $x$  and we get the result.
\item finally,
  \begin{eqnarray}
    \label{eq:caso3}
    p(y,0)=p(x,0)+1\quad\text{ and }\quad  p'(y,0)=p'(x,0)
  \end{eqnarray}
  or conversely  $p(y,0)=p(x,0)$ and $p'(y,0)=p'(x,0)+1$. We leave this case to the reader.
\end{itemize}

\smallskip

Finally, we prove the formula for $t>0$. The height difference
$H(x,0)-H(x,t)$ equals the number of particles that cross $x$ in the
time interval $[0,t]$, so
\begin{eqnarray}
  \label{eq:aaaa}
H(x,t)=H(x,0) - (p(x,t)-p(x,0)).  
\end{eqnarray}
 Then, \eqref{eq:aaa} follows from the fact that it holds for $t=0$ plus \eqref{eq:aaaa}.
\end{proof}
\begin{Theorem} [Stochastic domination]
\label{th:domina}
\label{th:stocdom} Let $\eta$ and $\eta'$ be two initial conditions in
$\Omega$ such that $ h_\eta(x)\leq h_{\eta'}(x)$ for every $x\in G^*$,
and denote $H(\cdot,t),H'(\cdot,t)$ the respective height functions
for the coupled evolutions that use the same Poisson process
realization $ W$. Then,
\[H(x,t)\leq H'(x,t) \; \text{ for every } x\in G^*, t\ge 0.\]

\end{Theorem}

\begin{proof}
Let $x\in G^*$ be such that
\[ H'(x,0)-H(x,0)=\min_y(H'(y,0)-H(y,0)).\]
We can assume without loss of generality that this minimum is $0$
(just by adding a constant to $h_{\eta}$, which will not affect the
dynamics at all) and that $x=0$.

Using  \eqref{eq:aaa} (at time $t=0$), the assumption $h_{\eta}(\cdot)\le h_{\eta'}(\cdot)$ easily implies $p(x,0)\ge p'(x,0)$ and therefore
\[ z_{(p,\ell)}(0)\leq z'_{(p,\ell)}(0) \text{ for every } \ell,p. \]
Then, it is easy to deduce that
$ z_{(p,\ell)} (t) \leq z'_{(p,\ell)}(t)$ for all later times. Indeed,
if $\xi \in \Xi_{(p,\ell),\eta',W,t}$ is such that
$x_0(\xi)\leq z_{(p,\ell)}(0)$ then $\xi $ belongs also to
$ \Xi_{(p,\ell),\eta,W,t}$ (property (III) in Definition \ref{def:Xi}
is guaranteed by the fact that all particles in $\eta$ are to the left
of their $\eta'$ counterparts, and the other properties are
obvious). The definition of the dynamics, Eq. \eqref{eq:3}, implies
that $ z'_{(p,\ell)} (t) \geq
z_{(p,\ell)}(t)$. % So, $z'_{(p,\ell)}$ cannot jump left of $z_{(p,\ell)}$
% without it also jumping at least to the same place, or even more to
% the left.
This inequality, combined with \eqref{eq:aaa} (this time at time $t$) gives us the desired domination. 
\end{proof}

\subsection{``Localizing'' the dynamics}
\label{sec:locdin}
It will be very useful, in the proof of Theorem \ref{th:hl}, to
consider a ``localized'' version of the dynamics where the Poisson
clocks are allowed to ring only in a certain finite subset of the
infinite lattice. Namely, fix $ \ell_-<\ell_+$ and
$ z_-< z_+$ consider a modified dynamics (that we distinguish
by a tilde) where the Poisson clocks rings $W_{(\ell,z)}$ for
$\ell\not\in ( \ell_-,\ell_+)$ or
$z\not\in [ z_-, z_+]$ are disregarded. 
In other words, the
modified dynamics is defined by the usual formula \eqref{eq:3} but the
Poisson realization $W=\{W_{(\ell,z)}\}_{\ell,z\in\mathbb Z}$ is
replaced by
$\tilde W=\{W_{(\ell,z)}\}_{ \ell_-< \ell<  \ell_+,  z_-\le
  z\le  z_+}$. Observe that the inequalities on $\ell$ are strict while those on $z$ are not.
\begin{Definition}
\label{def:D}
  Given $ \ell_-<\ell_+$ and
$ z_-< z_+$, let 
% \[
% D(\ell_-,\ell_+,z_-,z_+)=\{x\in G^*: \ell(x)\in ( \ell_-,\ell_+),  z(x)\in ( z_-, z_+)\}
% \]
% and 
\begin{eqnarray}
  \label{eq:DD}
  D(\ell_-,\ell_+,z_-,z_+)=\{x\in G^*: \bar\ell(x)\in [ \ell_-,\ell_+],  \bar z(x)\in [ z_-, z_+]\}  
\end{eqnarray}
with $ \bar\ell(x), \bar z(x)$ defined in \eqref{eq:elle}, \eqref{eq:zeta}.
\end{Definition}
We will refer to the above defined modified dynamics as to the ``dynamics localized in $D(\ell_-,\ell_+,z_-,z_+)$''. We have chosen a rectangular shape for the localization region $D$ just for simplicity.

We start from the following observation:
\begin{Proposition}
\label{claim:soff}
Let $\eta_1,\eta_2$ be two configurations in $\Omega$ such that
\begin{eqnarray}
  \label{eq:uguali}
h_{\eta_1}(x)=h_{\eta_2}(x) \text{  for every } x\in  D(\ell_-,\ell_+,z_-,z_+).
\end{eqnarray}
Couple the dynamics localized in $D(\ell_-,\ell_+,z_-,z_+)$,
started from $\eta_1,\eta_2$, by using the same realization $\tilde W$
for the Poisson clocks and call $\tilde H_i(\cdot,\cdot),i=1,2$ the 
corresponding height functions.  The following fact holds:
\begin{eqnarray}
  \label{eq:Htildeuguali}
  \tilde H_1(x,t)=\tilde H_2(x,t) \text{ for every } t\ge0, x\in  D(\ell_-,\ell_+,z_-,z_+).
\end{eqnarray}
\end{Proposition}
%The reason for the appearence of both $D$ and $\bar D$ in the statement is 
%a bit technical and will be explained along the proof.

\begin{proof}[Proof of Proposition \ref{claim:soff}] 

  Note first of all that condition \eqref{eq:uguali} is equivalent to the following statement (see Figure \ref{fig:6}):
  there is a particle on line $\ell_-\le \ell \le \ell_+$ with
  horizontal coordinate $ z_-<z< z_+$ for configuration
  $\eta_1$ iff there is a particle at the same location for
  configuration $\eta_2$. Therefore, possibly modulo changing
   the origin of the particle labels in one of the two
  configurations, we have that particle labelled
  $(p,\ell), \ell_-\le \ell\le \ell_+$ is at position
  $ z_-<z< z_+$ in $\eta_1$ iff the same happens for
  $\eta_2$. We denote $z^{(i)}_{(p,\ell)}(t)$ particle positions in
  the process $i=1,2$ (we omit the tildes to keep notations lighter).
\begin{figure}
\includegraphics[width=8cm]{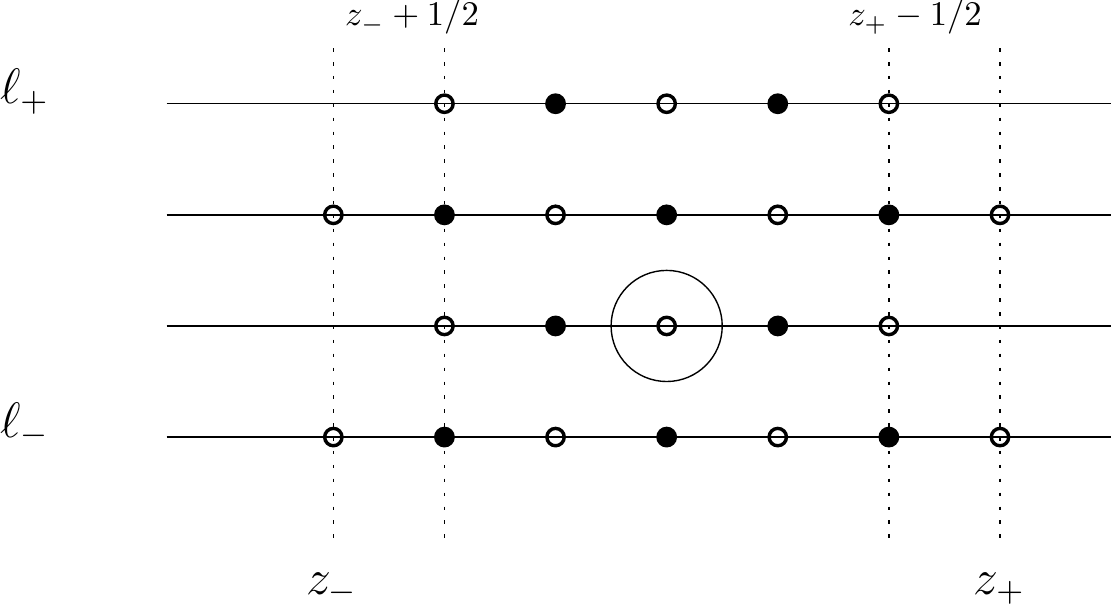}
  \caption{The white dots in the picture  (vertices of $G^*$, where the height is defined) comprise the rectangular set $D(\ell_-,\ell_+,z_-,z_+)$
of \eqref{eq:DD}. Black dots are vertices of 
 $G$ (possible particle positions) with horizontal coordinate strictly between $z_-$ and $z_+$ and vertical coordinate in $[\ell_-,\ell_+]$.
From \eqref{eq:grad2} we see that the height function on white white dots  determines particle occupation variables
on black dots. Vice versa, knowing the occupation variable of black dots and using \eqref{eq:height}, \eqref{eq:height2} and \eqref{eq:grad2} one can determine the height at white dots, once the height is fixed somewhere (say on the encircled vertex).}
\label{fig:6}
\end{figure}

  Observe that if a particle $(p,\ell)$ satisfies initially
  $z^{(i)}_{(p,\ell)}> z_+$, it is still possible that
  $ z_-\le z^{(i)}_{(p,\ell)}(t)\le z_+$ for some $t>0$; on the other
  hand, if $ z_-\le z^{(i)}_{(p,\ell)}(t)\le z_+$ for some $t\ge0$
  then the same property holds at later times, because there are no
  Poisson points of $\tilde W$ to the left of $ z_-$.

In view of this discussion, we see that the height $\tilde H_i(\cdot,t)$ in $D(\ell_-,\ell_+,z_-,z_+)$ at time $t$
is uniquely determined by the positions $ z^{(i)}_{(p,\ell)}(t)$  of particles 
$(p,\ell)$ with $\ell_-\le \ell\le \ell_+$ and  $z_-< z^{(i)}_{(p,\ell)}(t)<z_+$. 
  By definition of the localized dynamics, only particles with line
  index $ \ell_-<\ell< \ell_+$ can move, so we have to check
  \eqref{eq:Htildeuguali} only for $x$ such that
  $\bar \ell(x)\in (\ell_-,\ell_+)$.
The claim of the Proposition 
then follows if we can prove that, for every $(p,\ell)$ with $\ell_-<\ell<\ell_+$, we have
\begin{eqnarray}
  \label{eq:XiXi}
\Xi_{(p,\ell),\eta_1,\tilde W,t}=\Xi_{(p,\ell),\eta_2,\tilde W,t}.  
\end{eqnarray}
Let $\xi\in \Xi_{(p,\ell),\eta_1,\tilde W,t}$. In the Definition
\ref{def:Xi} of $\Xi_{(p,\ell),\eta,\tilde W,t}$, the initial
condition $\eta$ enters only through property (III).  Therefore, to prove
that $\xi\in \Xi_{(p,\ell),\eta_2,\tilde W,t}$ it suffices to show
that, if $((x,s),(p',\ell'))\in\xi$ and $x<z_{(p'',\ell'')}^{(2)}$ for
some $(p'',\ell'')\in I_{(p',\ell')}$ then
\begin{eqnarray}
  \label{eq:cs}
\text{there exists }
((x',s'),(p'',\ell''))\in \xi \text{  with } x'<x, s'<s.
\end{eqnarray}

Recall that, by the definition of $\tilde W$, one has $\ell_-<\ell'<\ell_+$ and
$ z_-\le x\le z_+$, so that $\ell_-\le \ell''\le \ell_+$ and
$z_-<z^{(2)}_{(p'',\ell'')}$.
% Let $((x,s),(p',\ell'))$ be one  such point.
% Eq. \eqref{eq:XiXi} follows if we show that:
% \begin{multline}
%   \label{eq:ifif}
% \{
% x<z_{(p'',\ell'')}^{(1)} \text{ for some } (p'',\ell'')\in
% I_{(p',\ell')}\}\\
% \Leftrightarrow\{
% x<z_{(p'',\ell'')}^{(2)} \text{ for some } (p'',\ell'')\in
% I_{(p',\ell')}\}
% \end{multline} 
If also $z^{(2)}_{(p'',\ell'')}<z_+$ then, as discussed above, we have
$z^{(1)}_{(p'',\ell'')}=z^{(2)}_{(p'',\ell'')}$ and \eqref{eq:cs}
follows because $\xi\in \Xi_{(p,\ell),\eta_1,\tilde W,t}$.  On the
other hand, if $z^{(2)}_{(p'',\ell'')}\ge z_+$ then the same holds for
$z^{(1)}_{(p'',\ell'')}$ (even if it is possible that
$z^{(1)}_{(p'',\ell'')}\ne z^{(2)}_{(p'',\ell'')}$). Given that
$x\le z_+$, we see that $x<z_{(p'',\ell'')}^{(i)}$ holds for both
$i=1,2$ (remark that equality cannot hold since $x$ and
$z_{(p'',\ell'')}^{(i)}$ differ at least by $1/2$, the corresponding
particles being on two neighboring lines) and again \eqref{eq:cs}
follows.

We have proven that $\Xi_{(p,\ell),\eta_1,\tilde W,t}\subset\Xi_{(p,\ell),\eta_2,\tilde W,t}$ and an analogous argument gives the opposite inclusion.
\end{proof}
 We have then a ``local'' version of Theorem
\ref{th:stocdom}, and it is actually this version we will mostly use:
 \begin{Theorem}[Stochastic domination: local version]
   \label{th:stocdomloc} Consider the dynamics localized in
   $D (\ell_-,\ell_+,z_-,z_+) $ defined as above.  Given two initial
   conditions $\eta,\eta'$ in $\Omega$ such that
   $ h_\eta(x)\leq h_{\eta'}(x)$ for every
   $x\in  D(\ell_-,\ell_+,z_-,z_+)$, denote
   $\tilde H(\cdot,t),\tilde H'(\cdot,t)$ the respective height
   functions at time $t$ for the coupled evolutions that use the same
   Poisson process realization $ \tilde W$. Then,
   \[H(x,t)\leq H'(x,t) \; \text{ for every } x\in
   D(\ell_-,\ell_+,z_-,z_+), t\ge 0.\]

 \end{Theorem}
% We can use this theorem in case of local domination, if no Poisson clock rings outside of the place where we have the domination. The only point in the proof that needs a change is where we check that $\xi \in \Xi_{(p,\ell),\eta,W,t}$, which is still true here since no Poisson point can end up in a site where $z_{(p,\ell)}(0) > z'_{(p,\ell)}(0)$  
 The proof is identical to that of Theorem \ref{th:stocdom}: indeed,
 by Proposition \ref{claim:soff}, we can without loss of generality imagine
 that the initial height functions verify $h_\eta(x)\leq h_{\eta'}(x)$
 for \emph{all} $x\in G^*$. At that point, the argument proceeds without any change.

\subsection{Propagation of information}

We have seen that whenever the initial condition $\eta$ is in
$\Omega$, the dynamics is well defined. Under a more restrictive
condition on $\eta$, roughly speaking if $z_{(p,\ell)}-z_{(p-k,\ell)}$
grows at most linearly with $k$, one has a stronger property:
information travels only ballistically through the system.  The way we
will use such property is to deduce that, with high probability, the
full dynamics and the dynamics localized in a large domain  have
exactly the same evolution, far from the boundary of the domain.
See Proposition
\ref{prop:propagation} below.

Let us define more precisely the condition on $\eta$.
Given $M\ge1$, let 
\begin{eqnarray}
\label{eq:Omegaeps}
\Omega_{M}=\{\eta\in\Omega:  z_{(p,\ell)}-z_{(p-k,\ell)}\le Mk \text{ for every } \ell,p,k\}.
\end{eqnarray}
% and
% \begin{eqnarray}
%   \label{eq:Omega*}
% \Omega^*_{M,j}=\{\eta\in\Omega: 
%  \text{ for every } n>j,|p|\le n,|\ell|\le n, k\ge1 \text{ one has } %k_0(p,\ell)<\infty \text{ such that: }
% \\ z_{(p,\ell)}-z_{(p-k,\ell)}\le \max(M k,\sqrt n) %% \text{ for every } k\ge k_0(p,\ell)
% \}.  
% \end{eqnarray}
\begin{Remark}
\label{rem:exM}
Under Assumption \ref{assumption:1}, there exists a finite $M$ such
that the initial condition $\eta^{(L)}$ belongs to $\Omega_M$ for
every $L$. Actually $M$ depends only on the set $A$ of Assumption
\ref{assumption:1}.
\end{Remark}

Informally, the ``ballistic propagation of information'' statement
says that, for initial conditions $\eta\in \Omega_M$, with high
probability the evolution of the height $H(x,t)$ at a fixed point $x$ and for times up to $T$
is not influenced by the realization of the Poisson processes at
points $y$ such that $|y-x|\ge T/\Delta$, for some positive constant
$\Delta=\Delta(M)$. A similar statement holds for typical initial configurations
sampled from a Gibbs measure $\pi_\rho$. Let us formalize these facts.

Given a realization $W$ of the Poisson process $\mathcal W$ and a
subset $\tilde W\subset W$, we will say that $W$ and $\tilde W$
coincide on
$[\bar z-n,\bar z+n]\times [\bar \ell-2n,\bar \ell+2n]\times [0,t]$ to
mean that $\tilde W$ contains all the points of
$W_{(\ell,z)}, z\in [\bar z-n,\bar z+n], \ell \in [\bar
\ell-2n,\bar\ell+2n]$,
up to time $t$. Moreover, $\tilde \eta(t)$ will denote the
configuration defined by particle positions \eqref{eq:3} with $W$
replaced by $\tilde W$, and $\tilde H(\cdot,\cdot)$ the corresponding
height function. The fact that $\tilde \eta(t)$ is a well-defined
configuration in $\Omega$ can be easily checked by noting that the
proof of Claims (1) and (2) of Proposition \ref{prop:bendefinito} required only
\emph{upper bounds} on the number of Poisson points of $W$ in certain
subsets.  It is also obvious from Definition \ref{def:Xi} that, if
$\tilde W\subset W$,
\[\Xi_{(p,\ell),\eta,\tilde W,t}\subseteq \Xi_{(p,\ell),\eta,W,t}\]
so that $\tilde z_{(p,\ell)}(t)\ge z_{(p,\ell)}(t)$ (particles move less quickly if there are fewer Poisson clock rings) and, as a consequence, 
\begin{eqnarray}
  \label{eq:primob}
H(x,t)\le \tilde H(x,t).  
\end{eqnarray}

\begin{Proposition}
  \label{prop:propagation} Let $M\ge1$, $\eta \in\Omega_M$ and
  $x\in G^*$.  There exists $c=c(M)>0$ and $\Delta=\Delta(M)>0$
  such that the following holds for every $n\ge1$, with probability at
  least $1-c e^{-n/c}$ w.r.t. the law $\mathbb P$ of the Poisson
  process $\mathcal W$.

  For every $\tilde W\subset W$ that coincides with $W$ on  \[R_n:=[\bar z-n,\bar z+n]\times (\bar \ell-2n,\bar \ell+2n)\times [0,\Delta n]\] (with $\bar z=\bar z(x),\bar \ell=\bar \ell(x)$ as in \eqref{eq:elle}, \eqref{eq:zeta}), one has
\begin{eqnarray}
H(x,t)=\tilde H(x,t), \; \forall t\leq \Delta n.
\end{eqnarray}

\end{Proposition}

% \begin{Definition} We call a \textit{partial process extracted from $\mathcal W$} a random variable $f(\mathcal W)$, such that almost surely $f(\mathcal W)\subseteq \mathcal W$. We will usually speak in terms of realization of such a process, and note $f(W)\subseteq W$. 

% \end{Definition}
% If $W$ is a realization of $\mathcal W$, on can define the dynamic using $f(W)$, the corresponding partial realization. 
\begin{Proposition}
\label{prop:propaeq}
  Let $\eta$ be sampled from $\pi_\rho$ and fix $x\in G^*$. There exists $c=c(\rho)>0$ and $\Delta=\Delta(\rho)>0$ such that 
the following holds for $n\ge1$, with probability at least $1-c e^{-n/c}$ w.r.t. the joint law $\pi_\rho\times \mathbb P$ of $\eta$ and $\mathcal W$:
for every $\tilde W\subset W$ that coincides with $W$ on  $R_n$, one has
\begin{eqnarray}
H(x,t)=\tilde H(x,t), \; \forall t\leq \Delta n.
\end{eqnarray}
% w
\end{Proposition}

 \begin{proof}
[Proof of Proposition   \ref{prop:propagation} ]
   At time zero, $H(x,0)=\tilde H(x,0)=h_\eta(x)$, so we need to show
   that the height variation is the same both when the Poisson point
   realization is $W$ or $\tilde W\subset W$ (actually, in view of
   \eqref{eq:primob}, only one bound is needed). The height at $x$
   changes if and only if a particle, located at time zero on line $\bar\ell:=\bar \ell(x)$ to the right of
   $\bar z:=\bar z(x)$, jumps to the left of $\bar z$. The claim of the
   Proposition follows if we prove that for a set of $W$ of
   probability at least $1- c e^{-n/c}$ the following happens:
\begin{enumerate}
\item [(i)] every particle that is on line $\bar \ell$, with initial position in $[\bar z, \bar z+n]$, has the same evolution in the time interval $[0,\Delta n]$, for the dynamics determined by $W$ and by any $\tilde W\subset W$ as above.
\item [(ii)] none of the particles that are on line $\bar \ell$, with initial position to the right of $\bar z+n$, 
jumps to the left of $\bar z$ up to time $\Delta n$.
\end{enumerate}
 To prove (i), 
 let $(p,\bar \ell)$ be such that $z_{(p,\bar \ell)} \in [\bar z,\bar z+n]$.
Since
\[ z_{(p,\bar\ell)}(t)= x_0(\xi_0)\le z_{(p,\bar\ell)} \text{ for a certain
    $\xi_0 \in \Xi_{(p,\bar \ell),\eta,W,t}$}, \] it suffices to show that, for every $t\le \Delta n$ and for every $\tilde W\subset W$ as
above, if  $\xi\in\Xi_{(p,\bar \ell),\eta,W,t}$
and $x_0(\xi)\le z_{(p,\bar\ell)}$ then $\xi\in \Xi_{(p,\bar\ell),\eta,\tilde W,t}$.
% for every $t\le \Delta n$ and for every $\tilde W\subset W$ as
% above.
Given that $\tilde W$ and $W$ coincide on $R_n$, it suffices to
prove that every point in 
$\xi$ has space-time
coordinates in $R_n$.  Given that $\Xi_{(p,\bar \ell),\eta,W,t}$ is
increasing in $t$, it suffices to prove the claim for $t=\Delta n$.
%$([x-n;x+n]\times [\ell-n;\ell_+n])\times [0;\Delta n]$.
 
Let us make the following choice: 
\begin{eqnarray}
  \label{eq:DE}
\Delta=\frac{e^{-3}}{64 M^2},\quad
 \epsilon = \frac{e^{-3}}{4\Delta n}  
\end{eqnarray}
and recall from \eqref{eq:BC} that
\begin{align}
  \label{eq:ennsm}
  \mathbb P(\exists \xi \in \Xi_{(p,\bar\ell),\eta,W,\Delta n}|\xi \subset [z_{(p-k,\bar\ell)},z_{(p-k,\bar\ell)}+\epsilon k^2)
\times[0,\Delta n], diam(\xi)\ge k)\leq Ce^{-k}
\end{align}
for some absolute constant $C$.
Therefore,  the probability that none of those events happens after rank $k_0$ is bigger than $1-2C e^{-k_0}$.
Let us choose 
\begin{eqnarray}
  \label{eq:k_0}
   k_0:= \left\lfloor \frac{n \sqrt{\Delta}}{e^{-3/2}}\right\rfloor,  
\end{eqnarray}
so that 
\begin{eqnarray}
  \label{eq:cl1}
 M k_0 \leq \epsilon k_0^2 \leq n. 
\end{eqnarray}

 From the definition of $\Omega_M$ we have %, for $n$ large enough, 
 \begin{eqnarray}
   \label{eq:casoM}
z_{(p,\bar\ell)}-z_{(p-k,\bar\ell)}\le M k \text{ for every $k$}    
 \end{eqnarray}
which implies that
\begin{eqnarray}
  \label{eq:cl2}
   k\geq k_0 \Rightarrow ~ z_{(p,\bar\ell)}-z_{(p-k,\bar\ell)}\le \epsilon k^2.
\end{eqnarray}
By Lemma \ref{lemma:astuto}, \eqref{eq:ennsm} and \eqref{eq:cl2}, except with probability $2C e^{-k_0}=2C e^{-n/c_0}$, the left-most point of 
any path in $\Xi_{(p,\bar\ell),\eta,W,\Delta n}$ with $x_0(\xi)\le z_{(p,\bar\ell)}$ is to the right of 
$z_{(p-k_0,\bar\ell)}$. 
Since by \eqref{eq:cl1} we have
$z_{(p,\bar\ell)}-z_{(p-k_0,\bar\ell)} \leq n$, we have proved that,
with probability at least $1-2C e^{-n/c_0}$, all points in
$\xi $ are to the right of
$z_{(p,\bar\ell)}-n$. By assumption
$x_0(\xi)\le z_{(p,\bar\ell)}\in[\bar z,\bar z+n]$ and therefore all points in $\xi$ have
horizontal coordinate between $\bar z-n$ and $\bar z+n$. Given this,
the fact that all points in $\xi$ have particle label $(q,\ell)$ with
$\ell\in [\bar \ell-2n,\bar \ell+2n]$ follows immediately from the
second claim in Lemma \ref{lemma:astuto}.

Since there can be at most $n$ particles on line $\bar \ell$ with
$z_{(p,\bar \ell)}\in[\bar z,\bar z+n]$, the statement of claim (i)
follows with probability at least $1-2 n Ce^{-n/ c_0}\ge 1-c e^{-n/c}$ for some $c>0$.

\medskip

Claim (ii) is proven similarly.  It is sufficient to prove the
statement for the left-most particle on line $\bar \ell$ to the right
of $\bar z+n$. Call its $(p_1,\bar\ell)$. By the same argument as
before, we have that, except with probability $2C e^{-n/c_0}$, the
left-most point of any path $\xi$ in
$\Xi_{(p_1,\bar\ell),\eta,W,\Delta n}$ such that $x_0(\xi)\le z_{(p_1,\bar\ell)}$ is to the right of
$z_{(p_1-k_0,\bar\ell)}$. Given that
$z_{(p_1,\bar \ell)}-z_{(p_1-k_0,\bar\ell)}\le n$ and
$z_{(p_1,\bar \ell)}\ge \bar z+n$,  $\xi$ is entirely to the right
of $\bar z$ as wished.  \end{proof}

\begin{proof}[Proof of Proposition \ref{prop:propaeq}]

  By translation invariance, let us assume that $x=(0,0)$ so that
  $\bar z(x)=-1/2,\bar \ell(x)=0$.  From Corollary \ref{cor:A1} there exists
  $M=M(\rho)<\infty, c=c(\rho)<\infty$ such that
\begin{multline}
  \label{eq:orti}
  \pi_\rho(E_{n,k_0}):=\pi_\rho(z_{(p,0)}-z_{(p-k,0)}\le M k \text{ for every } 1\le p\le n+1, k\ge k_0)\\\ge 1-c n(n\vee k_0) e^{-k_0}.
\end{multline}
We assume that $\eta$ satisfies condition $E_{n,k_0}$ with $k_0$
defined as in \eqref{eq:k_0} and \eqref{eq:DE} (note that in this case the pre-factor
$n(n\vee k_0)=O(n^2)$ in the r.h.s. of \eqref{eq:orti} is negligible
with respect to $e^{-k_0}$).  At this point, the proof proceeds very
similarly to that of Proposition \ref{prop:propagation}, because there
we used condition \eqref{eq:casoM} only for $k\ge k_0$ and not for
smaller values. 
\end{proof}

\section{Proof of the hydrodynamic limit before the appearance of shocks}

\label{sec:hydro1}

\subsection{The deterministic PDE: Proof of Proposition \ref{prop:PDE}}

\label{sec:PDE}
This is standard, but we give a sketchy proof for readers not
used to first-order non-linear PDEs. The PDE \eqref{eq:PDE} is solved
as usual by the method of characteristics \cite{Evans}. Let
$D v:\mathbb T\mapsto \mathbb R^2$ denote the differential of the
function $v(\cdot)$ defined in \eqref{eq:v}.  For
$x_0\in\mathbb R^2,t\ge0$ define $x(x_0,t)$ via
\begin{eqnarray}
  \label{eq:x0}
  x=x_0+t Dv(\nabla \phi_0(x_0)).
\end{eqnarray}
We claim first that, under Assumption \ref{assumption:2} on $\phi_0$, $x(\cdot,t)$
defines a global diffeomorphism of $\mathbb R^2$ for $t\le
T$, if $T>0$ is small enough. For this, notice  that the
differential w.r.t. $x_0$ of $x(x_0,t)$ is
\begin{eqnarray}
  \label{eq:diff}
 D_{x_0}x(x_0,t)= \mathbb I+t H_v(\nabla \phi_0(x_0))\cdot H_{\phi_0}(x_0)
\end{eqnarray}
where $\mathbb I$ is the $2\times 2$ identity matrix, $H_v$ is the
Hessian of the function $v:\mathbb T\mapsto \mathbb R$ and $H_{\phi_0}$
the Hessian of $\phi_0:\mathbb R^2\mapsto \mathbb R$. Since $\phi_0$ is
uniformly $C^2$, its gradient is uniformly away from $\partial\mathbb
T$ and $v(\cdot)$ is $C^\infty $ in the interior of $\mathbb T$, it
follows that the determinant of $D_{x_0}x(x_0,t)$ is in $(0,+\infty)$ and  bounded away from $0$ and $+\infty$
% \begin{eqnarray}
%   0<c_-< \|D_{x_0}x(\cdot,t)\|<c_+<\infty
% \end{eqnarray}
uniformly in $x_0$, for $t$ strictly smaller than $T_f$, the first time
where the r.h.s. of \eqref{eq:diff} is not invertible for some $x_0$.
\begin{Remark}
  \label{rem:Tunif}
  Note that the estimate on $T_f$ depends just on the estimate on the
  Hessian of $\phi_0$ and the distance of the range of $\nabla\phi_0$
  from $\partial\mathbb T$.
\end{Remark}
Also, we see that $|x(x_0,t)|\to\infty$ whenever $|x_0|\to\infty$
(simply because $D v(\nabla\phi_0(x_0))$ is uniformly bounded). Then we can
apply a theorem by Hadamard to deduce that $x(\cdot,t)$ is a global
diffeomorphism of $\mathbb R^2$:
\begin{Theorem}(\cite{Hadamard} and \cite[Th. A]{Gordon})
  A $C^1$ map $f$ from $\mathbb R^N$ to $\mathbb R^N$ is a diffeomorphism iff $f$ is proper (i.e. $|x|\to\infty$ implies $|f(x)|\to\infty$) and the Jacobian determinant $\det(\partial f_i/\partial x_j)$ never vanishes.
\end{Theorem}

We call $x_0(\cdot,t)$ the inverse of $x(\cdot,t)$.
Then, the solution of \eqref{eq:PDE} provided by the method of characteristics, namely
\begin{eqnarray}
  \label{eq:solPDE}
  \phi(x,t)=\phi_0(x_0(x,t))+t[D v(\nabla\phi_0(x_0(x,t)))\cdot\nabla\phi_0(x_0(x,t))
-v(\nabla \phi_0(x_0(x,t)))]
\end{eqnarray}
is twice differentiable in space and time, uniformly for
$x\in \mathbb R^2$ and $t<T_f-\epsilon$. This is easily checked:
indeed, \eqref{eq:solPDE} gives
\begin{eqnarray}
  \nabla \phi(x,t)=\nabla \phi_0(y)|_{y=x_0(x,t)}
\end{eqnarray}
(which by the way implies \eqref{eq:more}).
Differentiating once more  w.r.t. $x$ and using that the norm of
$D_x x_0(\cdot,t)$ is uniformly bounded for $t\le T_f-\epsilon$ (as can
be easily checked from \eqref{eq:x0}) and that $\phi_0$ is uniformly $C^2$, the uniform bound on the second space derivatives 
$\phi(x,t)$ follows. The bound on the second  time derivative is proven analogously.

\subsection{Proof of Theorem \ref{th:hl}}

\label{sec:hl}
\subsubsection{A few notations}

For lightness of notations we will  assume that $x=(0,0)$ in \eqref{eq:convergenza}. Let us fix $t<T_f$ and define
\begin{eqnarray}  \label{eq:Amax}
  \mathcal R=\sup_{x\in\mathbb R^2,s\le t} \|H_{\phi(x,s)}\|+|\partial^2_s\phi(x,s)|<\infty.
\end{eqnarray}
Recall from Remark \ref{rem:exM} that the initial condition verifies
$\eta^{(L)}\in \Omega_M$ for some finite $M$, uniformly in $L$. With Proposition \ref{prop:propaeq} in mind, define
\begin{eqnarray}
  \label{eq:MMM}
  M_0:=\max\{M, \max\{1/(1-\rho_1-\rho_2), \rho\in A\}\}<\infty
\end{eqnarray}
where $A$ is the subset of $\mathbb T$ that appears in Proposition \ref{prop:PDE}.
Let also $\Delta$ as in \eqref{eq:DE} with $M$ replaced by $M_0$.  We choose
$\epsilon >0$ as
\begin{eqnarray}
  \label{eq:orti2}
  \epsilon=\max\left\{\varepsilon\leq \frac{\delta \Delta}{3\mathcal R t(6+\Delta^2)}\;\text{ such that }\; K:=  \frac{t}{\varepsilon\Delta} \in \mathbb N \right\}
\end{eqnarray}
% \[ \left\{ \begin{aligned}  &\epsilon \leq \frac{\delta \Delta}{32At}\\
%                           &\epsilon \leq \frac{\delta}{8 \Delta A t} \\
%                            &K:=  \frac{t}{\epsilon\Delta} \in \mathbb N 
%    \end{aligned} \right. \] 
with $\delta$ as in the statement of Theorem \ref{th:hl} and we put
\begin{eqnarray}
  \label{eq:tau}
              \tau := \frac{t}{K}=\epsilon \Delta. 
\end{eqnarray}
As a first step we  suitably localize the dynamics, as in  Section \ref{sec:locdin}.
\begin{Definition}
\label{def:tildew}
  If $ W$ is the realization of Poisson processes that defines the
dynamics, we let $\tilde W$  be the sub-set of points of $W$ defined as follows:
a point of $\mathcal W_{(\ell,z)}$  of time coordinate $s\in (k\tau L,(k+1)\tau L],k=0,\dots,K-1$ belongs to $\tilde W$ iff
\begin{eqnarray}
  \label{eq:tfoto}
  z\in [z_-(k),z_+(k)], \ell\in (\ell_-(k),\ell_+(k)),
\end{eqnarray}
with 
\begin{eqnarray}
  \label{eq:zl}
  z_\pm(k)=\pm (2K-k)\frac{t}{K\Delta} L, \quad \ell_\pm(k)=\pm (4K-2k)\frac{t}{K\Delta} L.
\end{eqnarray}
\end{Definition}
% Finally, for $0\leq k \leq K$, we set
%  \begin{multline} 
% \label{eq:IIC}
% %&C:=[- K\epsilon , K\epsilon ]\times[-2 K\epsilon  ,2 K\epsilon ]\times[0,T]\\
% I_k:= [(-2 K+k)\epsilon   ,(2K-k) \epsilon ]\times[(-4K+2k)\epsilon , (4K-2k)\epsilon ]\\
% I_k^{(L)}=L I_k\\
% C:= \bigcup_{k=1}^{K} I^{(L)}_{k-1}\times[(k-1)\tau L,k\tau L].%\subset\left( G^*\times [0,TL]\right). %\cup C 
% \end{multline} 
Correspondingly, we let $\tilde H(\cdot,s)$ be the height function at
time $s$ for the evolution with $W$ replaced by $\tilde W$.  Note
that, with the conventions of Section \ref{sec:locdin}, the modified
dynamics is ``localized'' in $D(\ell_-(k),\ell_+(k),z_-(k),z_+(k))$ (recall Definition \ref{def:D}) in
the time interval $ (k\tau L,(k+1)\tau L]$. Remark that the rectangle
$D(\ell_-(k),\ell_+(k),z_-(k),z_+(k))$ shrinks as $k$ grows, but its
size is still of order $L$ for $k=K$: in fact,
\[
[z_-(K),z_+(K)]\times [\ell_-(K),\ell_+(K)]=[-tL/\Delta ,tL/\Delta ]\times
[-2tL/\Delta ,2tL/\Delta ].
\]
% fall in $C$, i.e. a Poisson point in 
% $\mathcal W_{(\ell,z)}$ with  time coordinate  $s$ belongs to $\tilde W$ if and only if 
% $(z,\ell,s)\in C$.
Thanks to Proposition \ref{prop:propagation}, we know that 
\[H(0,s)=\tilde H(0,s)\text{ for every } s\le tL \]
on an event
of probability at least $1-ce^{-{L}/{c}}$.
Therefore, it will be enough to prove \eqref{eq:convergenza} for  $\tilde H$ instead of $H$.

\subsubsection{Recursion}
\label{sec:recursion}

 With an eye on Definition \ref{def:D}, we let
 \begin{eqnarray}
   \label{eq:Dlk}
 D_k^{(L)}=D(\ell_-(k),\ell_+(k),z_-(k),z_+(k))\subset G^*. %: (\bar z(x),\bar \ell(x))\in  I^{(L)}_k\}.   
 \end{eqnarray}
We will prove by induction the following statement:
\begin{Proposition}
\label{prop:induttiva}
  Given $\delta>0$, for  $0\le k\leq K$
one has
\begin{eqnarray}
  \label{eq:ered}
 \lim_{L\to\infty}\mathbb P\left(\exists y \in D^{(L)}_k:  \frac1L \tilde H( y, k\tau L)-\phi(y/L,k \tau)>\frac{k+1}{K} \delta\right) = 0 
\end{eqnarray}
and
\begin{eqnarray}
  \label{eq:ered2}
 \lim_{L\to\infty}\mathbb P\left( \exists y \in D^{(L)}_k:  \frac1L \tilde H(y, k\tau L)-\phi(y/L,k \tau)<-\frac{k+1}{K} \delta\right) = 0 .
\end{eqnarray}
\end{Proposition}
Note that these statements for $k=K$ (taking $y=0$) imply the claim of
Theorem \ref{th:hl} at time $t$ (recall we are taking without loss of generality
$x=0$ in \eqref{eq:convergenza}, and the point $0$ is included in $D^{(L)}_K$).

\begin{proof}
  [Proof of Proposition \ref{prop:induttiva}] We will prove only
  \eqref{eq:ered}, the proof of \eqref{eq:ered2} being analogous. 
  Statement \eqref{eq:ered} is  true for $k=0$ (at time zero
  $\tilde H(\cdot,0)=h_\eta(\cdot)$ and the difference between
  $h_\eta(\cdot)/L$ and $\phi_0(\cdot/L)$ is deterministically $O(1/L)$,
  see \eqref{eq:condiniz}). We assume that \eqref{eq:ered} holds for
  some $k$ and prove it for $k+1$.  For every
  $u\in D^{(L)}_{k+1}$, we will show that
  \begin{eqnarray}
    \label{eq:provando}
\lim_{L\to\infty}\mathbb P\left( \frac1L \tilde H(  u, (k+1)\tau L)-\phi( u/L,(k+1) \tau)>\frac{k+3/2}{K} \delta\right) = 0.     
  \end{eqnarray}
Then, by a simple approximation argument we will obtain  \eqref{eq:ered} at level $k+1$.

Call $E_{k}$ the complementary of the event in parenthesis in
\eqref{eq:ered}. Suppose we are on the event
$\cap_{j=1}^k E_j$, whose probability is $1+o(1)$ as $L\to\infty$ (note that
$k\le K$ and $K$ does not grow with $L$). Since in particular we are
on event $E_k$, by monotonicity of the dynamics (Theorem
\ref{th:domina}) we can replace the height function 
\[\tilde H( y ,k\tau L), \,y\in D^{(L)}_k\]
by the higher height function
\begin{eqnarray}
  \label{eq:H'}
 H'( y ,k\tau L):=\lceil L \phi(y/L,\tau k)+L\frac {k+1}{K}\delta\rceil, \, y\in D^{(L)}_k.  
\end{eqnarray}
Starting from such configuration, we let the dynamics run in the time
interval $[k\tau L,(k+1)\tau L]$. Since in such time interval the dynamics is localized in  $D^{(L)}_k$  and we are interested in the height evolution inside $D^{(L)}_{k+1}\subset D^{(L)}_k $, by Proposition \ref{claim:soff} it is irrelevant how we define
$ H'(\cdot,k\tau L)$ outside $D_k^{(L)}$. For instance we can establish that \eqref{eq:H'} holds for every $y\in G^*$.

Recall from Proposition \ref{prop:PDE} that the gradient of
$\phi(\cdot,t)$ is in $A$ for all times, so that (cf. Remark
\ref{rem:exM}) the particle configuration with height function
$ H'(\cdot, k\tau L)$ is in $\Omega_M$, for the same $M$ as at time zero.  Therefore, by Proposition
\ref{prop:propagation} we can localize, in the whole time interval
$[k\tau L,(k+1)\tau L]$, the 
dynamics in the domain 
\begin{eqnarray}
  \label{eq:gato'}
D(u):=  D(\bar \ell( u)-2\epsilon L, \bar \ell( u)+ 2\epsilon L, \bar z( u)-\epsilon L, \bar z( u)+\epsilon L):
\end{eqnarray}
% Poisson clocks rings $W_{(\ell,z)}$ with
% $(\ell,z)$ outside the rectangle
% \[R_\epsilon(u):=[\bar z( u)-\epsilon L; \bar z( u)+ \epsilon L]\times[\bar \ell( u)-2\epsilon
% L; \bar \ell( u)+ 2\epsilon L]:\]
 except with a
probability $ce^{-L/c}$, the evolution of the
height at site $u$ for times $s\in [k\tau L,(k+1)\tau L]$ is
not affected (recall that $\tau=\epsilon \Delta$).
Note that $D(u)\subset  D^{(L)}_{k}$ because $ u\in D^{(L)}_{k+1}$: this is the reason why we defined the domains $D^{(L)}_{k}$ to be decreasing with $k$.

% For convenience, let 
% \[
% V_\epsilon( u)=\{y\in G^*:  (\bar z(y),\bar \ell(y))  \in R_\epsilon(u) \}\subset D^{(L)}_{k}
% \]
% (the inclusion is true because $ u\in D^{(L)}_{k+1}$: this is the reason why we defined the domains $D^{(L)}_{k}$ to be decreasing with $k$).

Now that we have localized the dynamics in the rectangle
$D( u)$, let us apply monotonicity (Theorem
\ref{th:domina}) once more and replace the  height function (at time
$k\tau L$)
\[
 H'(y,k\tau L), y\in G^*%, \, y\in V_\epsilon(\bar x)
\]
by a height function $H''(\cdot,k\tau L) $
defined as follows:
\begin{Definition}
The space gradients of the  function $H''(\cdot, k\tau L)$ have the same law as the gradient of $h_\eta$, with $\eta$ sampled from  the Gibbs measure $\pi_\rho$, with $\rho:= \nabla \phi ( u/L, k \tau)$. Moreover, 
the overall additive constant of the height function is fixed by the
condition
\begin{eqnarray}
\label{eq:add}
   H''(  u, k\tau L):= \lfloor L \phi(u/L,\tau k) + 6L \mathcal R \epsilon^2 + L \frac{k+1}{K} \delta\rfloor,
\end{eqnarray}
where we recall that $\mathcal R$ was defined in \eqref{eq:Amax}.
\label{def:H''}
\end{Definition}
% Let's compare $\tilde H'$ with two other height functions defined on $[x_1-\epsilon L; x_1+ \epsilon L]\times[x_2-\epsilon L; x_2+ \epsilon L]\times[kn\tau,(k+1)n\tau]  $ as followed:
As was the case for $H'$, also for $H''$ it is irrelevant how we define it outside the domain $D(u)$ where the dynamics has been localized; however, it is convenient to have  $H''(y, k\tau L)$  defined as above for every $y\in G^*$.

%\item[.]The spatial gradient of $\tilde H_1'$ is given by the height fluctuation of an equilibrium condition of slope $\rho:= \nabla \phi (x, k \tau)$
%\item[.]The evolution of  $\tilde H_1'$ is obtained following the same Poisson clocks as $\tilde H'$.
%\item[.] $\tilde H_2'(\lfloor xn\rfloor, k\tau n):= n \phi(x,\tau) - nA4 \epsilon ^2 - 2n \frac{k}{K} \delta$ and then is defined the same way.
We will prove at the end of the present section that with high
probability the height function $ H''$ is higher than $ H'$ in
the domain $D( u)$ of interest:
\begin{Lemma}
\label{lemma:Hche}
With probability going to $1$ as $L$ goes to infinity, we have
\begin{eqnarray}
  \label{eq:Hche}
  H'(y, k\tau L)\le H''(y,k\tau L) \quad \text{for every } y\in D( u).
\end{eqnarray}
% For all $y\in [x_1-\epsilon ; x_1+ \epsilon ]\times[x_2-\epsilon ; x_2+ \epsilon ]$:
% \[\tilde H_2'(\lfloor yn\rfloor, k\tau n)\leq \tilde H'(\lfloor yn\rfloor, k\tau n)\leq \tilde H_1'(\lfloor yn\rfloor, k\tau n)\]
\end{Lemma}
Summarizing what we discussed so far, we see that
\begin{multline}
 \mathbb P\left( \frac1L \tilde H(  u, (k+1)\tau L)-\phi( u/L,(k+1) \tau)>\frac{k+3/2}{K} \delta\right) \\\le
\mathbb P\left( \frac1L H''(  u, (k+1)\tau L)-\phi( u/L,(k+1) \tau)>\frac{k+3/2}{K} \delta\right)+\epsilon_L
\end{multline}
with $\lim_{L\to\infty}\epsilon_L=0$.% and we emphasize that the dynamics with height $H''$ is  .

At the end of this section we will prove the following statement, which implies the desired claim \eqref{eq:provando}:
\begin{Lemma}
Let $ H''(\cdot,s), s>k\tau L$ the height function at time $s$ for the dynamics localized in $D(u)$ in the time interval $[k\tau L,(k+1)\tau L]$, with
initial condition at time $k\tau L$ given by $H''(\cdot,k\tau L)$  as in Definition \ref{def:H''}. Then,
  \label{lemma:eqevol}
  \begin{eqnarray}
\lim_{L\to\infty}\mathbb P\left( \frac1L H''(  u, (k+1)\tau L)-\phi( u/L,(k+1) \tau)>\frac{k+3/2}{K} \delta\right)    =0.
  \end{eqnarray}
\end{Lemma}
Finally, let us show how the knowledge of  \eqref{eq:provando} for every $u\in D^{(L)}_{k+1}$ implies 
\eqref{eq:ered} at level $k+1$. 
In fact, for any fixed $\xi>0$ \eqref{eq:provando}  implies 
 \begin{gather}
 \label{eq:tvc}
\lim_{L\to\infty}\mathbb P\left(\forall y\in D^{(L)}_{k+1}\cap (\lfloor \xi L\rfloor \mathbb Z)^2, \frac1L \tilde H( y, (k+1)\tau L)-\phi( y/L,(k+1) \tau)\le \frac{k+3/2}{K} \delta\right) = 1, 
  \end{gather}
  simply because $D^{(L)}_{k+1}\cap (\lfloor \xi L\rfloor \mathbb Z)^2$
  contains a finite number (of order $\xi^{-2}$) of points $u$.  On the
  other hand, both the height function $\tilde H$ and $\phi$ are (deterministically) $1$-Lipschitz
  in space, so that in \eqref{eq:tvc} we can replace
  ``$\forall y\in D^{(L)}_{k+1}\cap (\lfloor \xi L\rfloor \mathbb Z)^2$'' with
  ``$\forall y\in D^{(L)}_{k+1}$'', provided we change $(k+3/2)\delta/K$
  into $(k+3/2)\delta/K+2\xi$. Choosing $\xi=\delta/(4K)$ gives
  \eqref{eq:ered} at level $k+1$.
\end{proof}

\begin{proof}[Proof of Lemma \ref{lemma:Hche}]
  Remark first of all that, if $y\in D( u)$, then
  $| u-y|\le 3\epsilon L$.  We know (cf. \eqref{eq:Amax}) that the second space derivatives of $\phi$
  are  bounded by $\mathcal R$ and therefore we have
  \begin{eqnarray}
    \label{eq:qui1}
 |\phi(y/L,k\tau)-\phi( u/L,k \tau)- \nabla \phi ( u/L, k \tau)\cdot(y-u)/L|\leq 5\mathcal R\epsilon^2.    
  \end{eqnarray}
From the definition of $H'$ we deduce that
\begin{eqnarray}
  \label{eq:H1}
   H'(y,k\tau L)\le \left\lceil L\phi( u/L,k \tau)+\nabla \phi ( u/L, k \tau)\cdot(y-u)+\frac {L(k+1)} {K} \delta+5\mathcal R L\epsilon^2\right\rceil.
\end{eqnarray}
On the other hand, from the properties of the Gibbs measure $\pi_\rho$ and more precisely from Lemma \ref{lemma:GFF} we deduce that, with probability $1+o(1)$,
\begin{multline}
  \label{eq:H2}
  H''(y,k\tau L)\ge  L\left[ \phi( u/L,\tau k)+\nabla \phi ( u/L, k \tau)
\cdot(y-u)/L +  6\mathcal R \epsilon ^2 +  \frac{k+1}{K} \delta\right]-(\log L)^2\\ \text{for every}\quad y\in D( u).
\end{multline}
The claim follows.
% We also know that with probability going to $1$, $\tilde H'$ is close to $\phi$, so we have with probability close to $1$:
% \[ \forall y, |\tilde H' (ny,nk\tau)-\tilde H'(nx,nk \tau)-n\rho(y-x)|\leq nA4\epsilon^2+n\frac{k}{K} \delta \] 
% Finally, we know from the result on the equilibrium that apart from an event of exponantially small (with $n$) probability:
% \[ \forall y, |\tilde H_1' (ny,nk\tau)-\tilde H_1'(nx,nk \tau)-n\rho(y-x)| \leq n\frac{k}{K}   \delta \]  % So with probability going to $1$:
% \[  \forall y, \tilde H_1'(ny,nk\tau) -\tilde H'(ny,nk\tau) \geq \tilde H_1'(nx,nk\tau) -\tilde H'(nx,nk\tau)-n2\frac{k}{K}\delta  -nA4\epsilon^2  \geq 0\]
\end{proof}
\begin{proof}[Proof of Lemma \ref{lemma:eqevol}] Recall that the
  height function $H''$ % in $V_\epsilon( u)$
  at time $k\tau L$ is chosen according to the equilibrium measure
  $\pi_\rho, \rho=\nabla \phi( u/L,k\tau)$, and the global additive
  constant is fixed by \eqref{eq:add}. We are interested in the height
  at time $(k+1)\tau L$ at site $ u$, which is at the center of the
  domain $D( u)$ where the dynamics is localized. By Proposition
  \ref{prop:propaeq}, the localized dynamics (in $D(u)$) and the full
  (i.e. non-localized) dynamics in the infinite lattice $G^*$ induce
  exactly the same height evolution at site $ u$ in the time interval
  $[\tau k L, (k+1)\tau L]$, except with exponentially small
  probability in $L$. On the other hand, for the dynamics on the
  infinite lattice with initial condition sampled from the stationary
  measure $\pi_\rho$ we can apply \eqref{eq:flutt} (say with
  $\eta=1/2$), which gives
  \begin{multline}
    \label{eq:ok}
\lim_{L\to\infty}    \mathbb P\left(\frac1L H''( u,(k+1)\tau L)>\phi( u/L,k\tau)-v(\nabla \phi( u/L,k\tau))\tau\right.
\\\left.+6\mathcal R \epsilon^2+\frac {k+1}{K}\delta
    + \frac1{\sqrt L}\right)=0.
  \end{multline}
Finally, from  smoothness in time of the solution $\phi(\cdot,\cdot)$ (cf. \eqref{eq:Amax}),
\begin{eqnarray}
  \label{eq:qui2}
|\phi( u/L,k\tau)-v(\nabla \phi( u/L,k\tau))\tau-\phi( u/L,(k+1)\tau)|\le \mathcal R \tau^2.  
\end{eqnarray}
Then, the statement of the Lemma follows provided that
\begin{eqnarray}
  \label{eq:provide}
  6\mathcal R \epsilon^2+\frac {k+1}{K}\delta+L^{-1/2}+\mathcal R\tau^2\le \frac{k+3/2}{K}\delta.
\end{eqnarray}
We leave to the reader to check that this is guaranteed (for $L$ large) by the choices of parameters we made in \eqref{eq:orti2} and \eqref{eq:tau}.
\end{proof}

\begin{Remark}
\label{rem:stripp}
  Note that we did not really need the fact that $\phi(\cdot,\cdot)$ is
  $C^2$ with respect to space and time, but rather that its space
  gradient $\nabla \phi(\cdot,t)$ is a Lipschitz function of space (this was used in
  \eqref{eq:qui1}) and that $\partial_t \phi(x,\cdot)$ is Lipschitz in
  time (this was used in \eqref{eq:qui2}), with Lipschitz constants bounded by
  $\mathcal R<\infty$ up to time $T$.
\end{Remark}

\section{Hydrodynamics with shocks: Proof of Theorem \ref{th:shock}
  and Proposition \ref{th:guillaume}}
\label{sec:shocks}

The proof of Theorem \ref{th:shock} consists of an easy lower
bound, Proposition \ref{prop:lowb}, and of a more subtle upper
bound, Proposition \ref{prop:uppb}.  Fortunately, we will see that
most of the work needed for the upper bound has already been done in
the proof of Theorem \ref{th:hl}.

\subsection{Lower bound}
We start by proving:
\begin{Proposition}
\label{prop:lowb}
  For every $x\in \mathbb R^2,t>0,\delta>0$,
  \begin{eqnarray}
    \label{eq:lowb}
    \lim_{L\to\infty}  \mathbb P\left(\frac1L H(\lfloor xL\rfloor, tL)<\phi(x,t)-\delta\right)=0.
  \end{eqnarray}
\end{Proposition}
\begin{proof}[Proof of Proposition \ref{prop:lowb}]
This is the easy bound, since 
\begin{eqnarray}
  \phi(x,t)=\max_{\rho\in\mathcal A}\{\phi_\rho(x,t)\}, \quad \phi_\rho(x,t)=\rho\cdot x-v(\rho)t-\phi_0^*(\rho)
\end{eqnarray}
% where
% \begin{eqnarray}
% \phi_u(x,t)=c \,x\cdot\beta+\psi_u({ n}\cdot x,t), \quad \psi_u(y,t)=V(u)t+y u
% \end{eqnarray}
and then it is sufficient to prove that
 \begin{eqnarray}
    \label{eq:lowb2}
    \lim_{L\to\infty}  \mathbb P\left(\frac1L H(\lfloor xL\rfloor, tL)<\phi_\rho(x,t)-\frac\delta2\right)=0
  \end{eqnarray}
  for every $\rho\in \mathcal A$ to conclude
  easily.  On the other hand, to prove \eqref{eq:lowb2} we can replace
  (by monotonicity) the initial condition $\eta^{(L)}$ by a lower
  initial condition $\tilde\eta^{(L)}$ with height function
\begin{eqnarray}
\label{eq:condiniz3}
h_{\tilde \eta^{(L)}}(x)=\lfloor L \phi_\rho(x/L,0)\rfloor \quad \text{for every }\quad x\in \mathbb Z^2.
\end{eqnarray}
Then, \eqref{eq:lowb2} is implied by Theorem \ref{th:hl}, since the
PDE \eqref{eq:PDE} with initial condition $\phi_\rho(\cdot,0)$ has smooth
(and actually affine) solution $\phi_\rho(\cdot,t)$ for all times.
  \end{proof}

\subsection{Upper bound} 
It remains to show:
\begin{Proposition}
\label{prop:uppb}
  For every $x\in \mathbb R^2,t>0,\delta>0$,
  \begin{eqnarray}
    \label{eq:uppb}
    \lim_{L\to\infty}  \mathbb P\left(\frac1L H(\lfloor xL\rfloor, tL)>\phi(x,t)+\delta\right)=0.
  \end{eqnarray}
\end{Proposition}

\begin{proof}[Proof of Proposition \ref{prop:uppb}]
  We begin with a few  considerations about the PDE
  \eqref{eq:PDE} and its viscosity solution \eqref{eq:solvisc}.  First
  of all, by the involutive property of the Legendre-Fenchel transform
  when acting on the convex function $\phi(\cdot,t)$, we can trivially
  rewrite
\begin{eqnarray}
  \label{eq:ie}
  \phi(x,t)=[w_t]^*(x)
  %\max_{y\in\mathbb R^2}\{y\cdot x-\bar v^{**}(y)t\},\quad t>0.
\end{eqnarray}
where
\begin{eqnarray}
  \label{eq:wt}
  w_t(y):=[t v+\phi_0^*]^{**}(y)
\end{eqnarray}
that is nothing but the lower convex envelope of $tv(\cdot)+\phi_0^*(\cdot)$.

Next, given $\delta_1>0$, let $\psi^*:\mathbb R^2\mapsto \mathbb R$ be a convex function such that 
\begin{eqnarray}
\label{eq:delta1}
  -\delta_1\le \psi^*(y)\le 0, y\in \mathbb T,
\end{eqnarray}
while at the same time $\psi^*$ is smooth and in particular its
Hessian $H_{\psi^*}$ satisfies $H_{\psi^*}\ge \epsilon \mathbb I$
everywhere ($\mathbb I$ being the $2\times 2$ identity matrix), for
some constant $\epsilon(\delta_1)>0$.
% \begin{eqnarray}
%   \label{eq:deriv}
% \nabla\hat\phi_0(x)\in \mathcal C(\rho_1,\dots,\rho_k) \quad \text{for every } x.
% \end{eqnarray}
For instance, one can take $\psi^*(\cdot)$ to be paraboloid with suitably chosen parameters. % the convolution of
% $\phi_0 $ with the heat kernel at some small time $T=T(\delta_1)$.
% %We know that the solution of ... with initial condition $\tilde\psi_0$ is smooth up and including some positive time $t_0$. 
We then define
\begin{eqnarray}
  \label{eq:phipsi}
  \hat \phi(x,t):=[w_t+\psi^*]^*(x)=\sup_y\{y\cdot x-(w_t(y)+\psi^*(y))\},
\end{eqnarray}
to be compared with \eqref{eq:ie}.
Equations
\eqref{eq:delta1}, \eqref{eq:ie}, together with the fact that the supremum in and \eqref{eq:phipsi} can be restricted to $y\in \overline {\mathcal A}\subset \mathbb T$, imply that
\begin{eqnarray}
\label{eq:mp}
 0\le \hat \phi(x,t)-\phi(x,t)\le   \delta_1\quad \forall x,t.
\end{eqnarray}
% because
% \[-\delta_1=\phi^*_0(\cdot)-\delta_1\le \hat\phi^*_0(\cdot)\le
%   \phi^*_0(\cdot)=0\] on $\mathcal C(\rho_1,\dots,\rho_k)$. 

Finally, we observe the following:
\begin{Lemma}
\label{lemma:t0}
There exists $\tau>0$ (depending on the choice of $\psi^*$ and in
particular on the parameter $\delta_1$ in \eqref{eq:delta1}) such
that, for every $s\ge0$, the solution $\phi_s(x,t)$ of Eq.
\eqref{eq:PDE} with initial condition
$\phi(\cdot,0):=\hat\phi(\cdot,s)$ is smooth on the time interval
$[0,\tau]$. More precisely, as long as $t\in[0,\tau]$, the space gradient
$\nabla\phi_s(x,t)$ is Lipschitz with respect to $x$ while $\partial_t\phi_s(x,t)$
is Lipschitz w.r.t. $t$; the Lipschitz constants are uniform w.r.t. $s$.
\end{Lemma}
\begin{proof}
  [Proof of Lemma \ref{lemma:t0}]
Since the initial condition $\hat\phi(\cdot,s)$ is convex, the solution of \eqref{eq:PDE} can be written as 
\begin{eqnarray}
  \phi_s(x,t)=[t v+\hat\phi(s)^*]^*=[t v+w_s+\psi^*]^*=:[G_{s,t}]^*.
\end{eqnarray}
Given that $w_s(\cdot)$ is convex and $\psi^*(\cdot)$ is strictly
convex with Hessian lower bounded by $\epsilon \mathbb I$ we deduce
that, for $t\le \tau$ with $\tau $ small enough, the function
$G_{s,t}(\cdot)$ is strictly convex and moreover
\begin{eqnarray}
\label{eq:rotonda}
  G_{s,t}(y_2)-G_{s,t}(y_1)-\nabla G_{s,t}(y_1)\cdot (y_2-y_1)\ge \epsilon \|y_1-y_2\|^2/4.
\end{eqnarray}
Strict convexity of $G_{s,t}(\cdot)$ implies differentiability of
$\phi_s(\cdot,t)$ \cite[Th. 11.13]{cf:Rocka}. The spatial gradient
$\nabla \phi_s(x,t)$ is the unique point $z$ that realizes the
supremum in the Legendre-Fenchel transform \eqref{eq:Leg}, with
$f\equiv G_{s,t}$. Of course
$\nabla \phi_s(x,t)\in \overline{\mathcal A}$ because
$G_{t,s}=+\infty$ outside $\overline{\mathcal A}$ (recall that
$\mathcal A$ is the range of the sub-differential of $\phi_0(\cdot)$).

The claim on the Lipschitz continuity of $\nabla \phi_s(x,t)$ with
respect of the space variable, with Lipschitz constant depending on
$\epsilon$, then easily follows from \eqref{eq:rotonda} and the
definition of Legendre-Fenchel transform (we skip elementary details).
The proof of Lipschitz continuity of $\partial_t\phi_s(x,t)$ 
w.r.t. $t$ is similar. Indeed, given that
$\partial_t\phi_s(x,t)=v(\nabla\phi_s(x,t)) $ (because $\phi_s(\cdot,\cdot)$ solves
\eqref{eq:PDE}) and $v(\cdot)$ is smooth in $\overline{\mathcal A}$, the
proof reduces to proving that $\nabla \phi_s(x,t)$ is Lipschitz
w.r.t. time and once more this follows easily from \eqref{eq:rotonda}.
\end{proof}

% We remark:
% \begin{Lemma}
%   \label{lemma:smoothhat}
%   The function  $\hat\phi(\cdot,\cdot):\mathbb R^2\times \mathbb R^+\mapsto \mathbb R$ is   differentiable.
% and $\sup_{x,t}\|H_{\hat\phi(x,t)}\|<\infty$, where as usual $H_{\hat\phi(x,t)}$ is the Hessian (w.r.t. the space variable) of $\hat\phi(\cdot,t)$ computed at $x$.
% \end{Lemma}
% In fact, by the assumption on $\psi^*$ and convexity of $w_t$ we have
% for the Hessian of $w_t+\psi^*$ that
% $H_{w_t+\psi^*}\ge \epsilon \mathbb I$ and by duality
% \cite[Sec. 11.C]{cf:Rocka} we deduce
% $H_{\hat \phi(\cdot,t)}\le (1/\epsilon)\mathbb I$ because
% $\hat\phi(\cdot,t)$ is  the Legendre-Fenchel transform of $w_t+\psi^*$.

% One has  also that 
% \begin{eqnarray}
%   \label{eq:deriv2}
% \nabla\hat\phi(x,t)\in \overline{\mathcal A}\subset \stackrel\circ{\mathbb T} \quad \forall x,t,
% %  u_-\le \partial_y\hat \psi_0(y,t)\le u_+\quad \forall y,t.  
% \end{eqnarray}
% because $\nabla\hat\phi(x,t)$ is  the value of $y$ that realizes the supremum in \eqref{eq:phipsi} and $w_t$  is $+\infty$ outside $\overline{\mathcal A}$ 

% %this holds at time zero and  the space gradient is constant along characteristic lines.
% %The maximum principle, or more explicitly %\eqref{eq:PDEbar},
% Finally, from Remark \ref{rem:Tunif}, Eq. \eqref{eq:deriv2} and Lemma
% \ref{lemma:smoothhat} we have the following:

\medskip We will prove \eqref{eq:uppb} at time $t$ and we assume for
lightness of notations that $x=0$.  The proof uses a
recursion that is very similar to that employed in Section
\ref{sec:recursion}; therefore, we try to use as far as possible the same
notations as we used there, and we give fewer details here.

We break the time interval $[0,t]$ into sub-intervals
$[(k-1),k\tau)$, with $\tau$ as in Lemma \ref{lemma:t0} and
$k\le K:=t/\tau$ (we assume for simplicity that $K\in\mathbb N$).
% \begin{eqnarray}
%   \label{eq:phitilde}
% \end{eqnarray}
% Defining as usual
% \begin{eqnarray}
%   \hat\phi(x,t)=c\, x\cdot \beta+ \hat \psi(x\cdot { n},t),
% \end{eqnarray}
% thanks to the uniform bounds \eqref{eq:unifB} and \eqref{eq:deriv2} on
% the first and second space derivative of $\hat \psi$, we know that
% we have immediately that, uniformly in time and space, the space
% second derivatives of $\hat\phi(x,t)$ are bounded and the gradient
% $\nabla\hat\phi$ is bounded away from $\partial\mathbb T$. Then, from
For $1\le k\le K-1$ we define $\phi^{(k)}(\cdot)$ to be the
  solution at time $\tau$ of the PDE \eqref{eq:PDE} with initial
  condition $\hat \phi(\cdot,(k-1)\tau)$.

  % The viscosity solution \eqref{eq:solvisc} is clearly increasing w.r.t. the initial condition $\phi_0(\cdot)$ and decreasing w.r.t. the function $v(\cdot)$. Therefore, from the inequalities
  % $\phi(\cdot)\le \hat\phi_0(\cdot)$ and $\bar v(\cdot)\ge  \bar v^{**}(\cdot)$ we see
  We will prove  at the end of this section:
  \begin{Lemma}
    \label{lemma:lele}
    For any $k\le K$ one has
\begin{eqnarray}
  \label{eq:lele}
 \phi(\cdot,k \tau)\le \phi^{(k)}(\cdot)\le \hat \phi(\cdot,k \tau).
\end{eqnarray}
  \end{Lemma}
 From this and \eqref{eq:mp} we deduce that
\begin{eqnarray}
\label{mlc}
  \|\phi(\cdot,k \tau)- \phi^{(k)}(\cdot)\|_\infty\le \delta_1
\end{eqnarray}
for every $k$.

The initial configuration $\eta^{(L)}$ defined in \eqref{eq:condiniz} with $\phi_0$ as in Theorem \ref{th:shock} belongs to $\Omega_M$ for some
finite $M$, uniformly in $L$, simply because the gradient of
$\phi_0(\cdot)$ is bounded away from $\partial \mathbb T$.  We replace
the Poisson point realization $W$ that defines the dynamics with
$\tilde W$ as in Definition \ref{def:tildew}: in other words, in the
time interval $[k\tau L,(k+1)\tau L]$ we are localizing the dynamics
in the rectangle $D^{(L)}_k$ defined in \eqref{eq:Dlk}. From
Proposition \ref{prop:propagation} we know that the resulting height
function $\tilde H(0,s)$ is the same as $H(0,s) $ for every $s\le tL$,
except with a probability going to zero exponentially as $L\to\infty$.

At time zero we replace, by monotonicity, the initial profile
$\phi_0(\cdot)$ in the assumption of Theorem \ref{th:shock} by
$\hat\phi(\cdot,0)= [\phi_0^*+\psi^*]^*(\cdot)\ge \phi_0(\cdot)$.
% Also, in analogy with \eqref{eq:IIC}, set 
% \begin{multline}
%   I_i=[(-2M+i)c,(2M-i)c]\times [(-4M+2i)c,(4M-2i)c]\\
% I^{(L)}_i=L I_i\\
% C=\cup_{i=1}^MI^{(L)}_{i-1}\times [(i-1)T_0L ,i T_0 L]\subset G^*\times [0,T]\\
% D_i^{(L)}=\{x\in G^*:(\bar z(x),\bar \ell(x))\in I_i^{(L)}\}.
% \end{multline}
% By proposition \ref{prop:propagation}, we know that we can choose
% $c>0$ such that, except with probability $o(1)$ as $L\to\infty$, the
% evolution of the height $H(0,t)$ at point $0$ in the time interval
% $[0,T]$ is not modified if we ignore the Poisson clock rings outside
% $C$, i.e. the rings of $\mathcal W_{(\ell,z)}$ at times $s$ such that
% $(z,\ell,s)\not\in C$. The choice of the constant $c$ depends on $T_0$
% and on the distance between $\partial \mathbb T$ and the range of the
% gradient of $\phi_0$; in other words, $c=c(T_0,u_-,u_+,n,\beta)$.
% Since in \eqref{eq:uppb} we are after an upper bound on the height
% function, at time $0$ we can replace by monotonicity the initial
% profile $\phi_0(\cdot)$ in the statement of Theorem \ref{th:shock}
% with the higher profile $\hat\phi(0,\cdot)\equiv \phi^{(0)}(\cdot)$.
Then, we run the dynamics for a time $L \tau$. At time $L \tau$ we
know by Theorem \ref{th:hl} that the height function is close to
$\phi^{(1)}(\cdot)$ and in particular lower than
$\phi^{(1)}(\cdot)+\delta_2$, for any fixed $\delta_2>0$.  More
precisely,
\begin{eqnarray}
  \label{eq:mopr}
  \lim_{L\to\infty}\mathbb P\left(
\exists y\in D^{(L)}_1, \frac1L \tilde H(y, \tau L)>\phi^{(1)}(y/L)+\delta_2
\right)=0.
\end{eqnarray}
Thanks to  Remark \ref{rem:stripp}, we are guaranteed that Theorem
\ref{th:hl} can be applied here because from Lemma \ref{lemma:t0} we
know that the solution of the PDE \eqref{eq:PDE} is sufficiently
smooth (its gradient  is Lipschitz in space and its
time derivative  is Lipschitz in time) up to  time $\tau$.

At time $\tau L$, we replace (again, by monotonicity) the height function in $D^{(L)}_1$ by the discretization 
\[y\in D^{(L)}_1\mapsto 
\lfloor L(\hat\phi(\cdot/L,\tau)+\delta_2)\rfloor\ge \lfloor L(\phi^{(1)}(\cdot/L)+\delta_2)\rfloor
\]
 and we run the 
evolution for another time interval $\tau L$ (in such time interval $[\tau L,2\tau L]$ the dynamics is localized in $D_1^{(L)}$). 

 Repeating the procedure $j$ times, we obtain at time $\tau j L$ 
\begin{eqnarray}
  \label{eq:mopr2}
  \lim_{L\to\infty}\mathbb P\left(
\exists y\in D^{(L)}_j: \frac1L \tilde H(y, \tau j L)>\phi^{(j)}(y/L)+j\delta_2
\right)=0
\end{eqnarray}
For $j=K$ and taking $x=0\in D_K^{(L)}$ we see that
\begin{eqnarray}
  \label{eq:mop3r}
  \lim_{L\to\infty}\mathbb P\left(
 \frac1L \tilde H(0, t L)>\phi^{(K)}(0)+K\delta_2
\right)=0. 
\end{eqnarray}
Using \eqref{mlc}, recalling $K\tau=t$ and taking $\delta_1,\delta_2$ sufficiently small so that
\begin{equation}
  \label{eq:ddd}
\delta_1+K\delta_2<\delta  
\end{equation}
we obtain the statement \eqref{eq:uppb} with $x=0$ (recall that we chose $x=0$ just for lightness of notation).
Note that the choice \eqref{eq:ddd} is possible since $K=t/\tau$ and $\tau$ depends on $\delta_1$ but not on $\delta_2$.
\end{proof}

\begin{proof}
  [Proof of Lemma \ref{lemma:lele}]
The first inequality in \eqref{eq:lele} is equivalent to
\begin{eqnarray}
  \label{eq:lele1}
[  k\tau v+\phi_0^*]^*(\cdot)\le [\tau v+((k-1)\tau v+\phi_0^*)^{**}+\psi^*]^*(\cdot)
\end{eqnarray}
which holds because $\psi^*(\cdot)\le0$ (where $\phi^*_0$ is not $+\infty$) and $((k-1)\tau v+\phi_0^*)^{**}(\cdot)\le
((k-1)\tau v+\phi_0^*)(\cdot)$.

As for the second inequality in \eqref{eq:lele}, it is equivalent to
\begin{eqnarray}
  \label{eq:lele2}
   [\tau v+((k-1)\tau v+\phi_0^*)^{**}+\psi^*]^*(\cdot)\le [(k\tau v+\phi_0^*)^{**}+\psi^*]^*(\cdot)
\end{eqnarray}
which follows if we can prove that, for every $y$,
\begin{eqnarray}
  \label{eq:ml}
  \tau v(y)+((k-1)\tau v+\phi_0^*)^{**}(y)\ge (k\tau v+\phi_0^*)^{**}(y)
\end{eqnarray}
(it suffices to prove this for $y\in \overline{\mathcal A}$ since outside $\overline{\mathcal A}$ both sides of the inequality are $+\infty$).
For lightness of notation we put
\begin{eqnarray}
  \label{eq:cfg}
  c:=k-1,\quad f(\cdot):=\tau v(\cdot),\quad g(\cdot):=\phi_0^*(\cdot).
\end{eqnarray}
To prove \eqref{eq:ml} on can proceed as follows. The double
Legendre-Fenchel transform in the l.h.s. is equivalently given by
\cite[Prop. 2.31 and Th. 11.1]{cf:Rocka}
\begin{eqnarray}
  \label{eq:rock1}
  (cf+g)^{**}(y)=\inf_{\{\lambda_i\}, \{y_i\}} \Bigl(\sum_{i=1,2}\lambda_i [c f(y_i)+g(y_i)]\Bigr),
\end{eqnarray}
where the infimum is taken over $\lambda_i\ge0$ and $y_i\in \mathbb R^2,i=1,2$ such that
$ \sum_{i=1}^2\lambda_i=1 $ and $\sum_{i=1}^2 \lambda_i y_i=y$. Since
$cf(\cdot)+g(\cdot)=+\infty$ outside the compact set $\overline{\mathcal A}$, the
infimum is realized for some values
$\{\tilde\lambda_i,\tilde y_i\}_{i=1,2}$.  Note that
\begin{eqnarray}
  \label{eq:ml2}
\sum_{i=1}^2 \tilde\lambda_if(\tilde y_i)\le f(y).  
\end{eqnarray}
  In fact, in the opposite case we would have
  \begin{gather}
    (c f+g)^{**}(y)=  \sum_{i=1,2}\tilde \lambda_i [c f(\tilde y_i)+g(y_i)]
    \\>c f(y)+\sum_{i=1,2}\tilde \lambda_ig(y_i)\ge c f(y)+g(y)
  \end{gather}
  (in the last step we used convexity of $g$) so that $(cf+g)^{**}(y)>(cf+g)(y)$ which is false.
  Putting everything together we have
  \begin{multline}
    [(1+c)f+g]^{**}(y)=\inf_{\{\lambda_i\}, \{y_i\}}  \Bigl(\sum_{i=1,2}\lambda_i [(1+c) f(y_i)+g(y_i)]\Bigr)\\
    \le
    \sum_{i=1}^2\tilde\lambda_i[(1+c)f(\tilde y_i)+g(\tilde y_i)]\\\le f(y)+
    \sum_{i=1}^2\tilde\lambda_i[cf(\tilde y_i)+g(\tilde y_i)]=f(y)+
    (cf+g)^{**}(y)
  \end{multline}
  where in the second inequality we have used \eqref{eq:ml2}. In view of \eqref{eq:cfg}, this is exactly the desired inequality \eqref{eq:ml}.
\end{proof}

To conclude, let us prove Proposition \ref{th:guillaume}. Since the
Hessian of $v(\cdot)$ has mixed signature everywhere in $\stackrel\circ{\mathbb T}$ and
$\mathcal A\subset \stackrel\circ{\mathbb T}$ has non-empty interior,
$t v(\cdot)+\phi_0^*(\cdot)$ is not convex on $\mathcal A$ for $t$ sufficiently large and
therefore $w_t(\cdot)=[t v+\phi_0^*]^{**}(\cdot)$ is not strictly convex there. On the
other hand $\phi(\cdot,t)$ is the Legendre-Fenchel transform of
$w_t(\cdot)$ (cf. \eqref{eq:ie}) and it is known that the
Legendre-Fenchel transform of a non-strictly convex function cannot be
differentiable everywhere \cite[Th. 11.13]{cf:Rocka}.

\section*{Acknowledgments}
We are grateful to Guillaume Aubrun, Christophe Bahadoran, Alexei
Borodin and Hubert Lacoin for very valuable discussions.  F.  T.  was
partially funded by the ANR-15-CE40-0020-03 Grant LSD, by the CNRS
PICS grant ``Interfaces al\'eatoires discr\`etes et dynamiques de
Glauber'' and by MIT-France Seed Fund ``Two-dimensional Interface
Growth and Anisotropic KPZ Equation''.

\end{document}